\documentclass[a4paper,english,openright,11pt]{smfartVScomptage}

\usepackage{smfenum,smfthm,amsmath,amssymb,amscd,mathrsfs,euscript,color}
\usepackage{stmaryrd,mathabx,upgreek}
\usepackage[latin1]{inputenc}
\input{xypic}

\numberwithin{equation}{section} 
\theoremstyle{plain}

\newcounter{nonum}

\def\CC{\mathbf{C}}
\def\NN{\mathbf{N}}
\def\QQ{\mathbf{Q}}

\def\ZZ{\mathbf{Z}}

\def\B{{\rm B}}

\def\D{{\rm D}}
\def\E{{\rm E}}
\def\F{{\rm F}}
\def\G{{\rm G}}
\def\H{{\rm H}}

\def\J{{\rm J}}
\def\K{{\rm K}}
\def\L{{\rm L}}
\def\M{{\rm M}}
\def\N{{\rm N}}
\def\P{{\rm P}}
\def\Q{{\rm Q}}
\def\R{{\rm R}}

\def\U{{\rm U}}

\def\W{{\rm W}}
\def\X{{\rm X}}
\def\Y{{\rm Y}}
\def\Z{{\rm Z}}

\def\Cc{\EuScript{C}}
\def\Dd{\EuScript{D}}
\def\Ee{\EuScript{E}}

\def\Oo{\EuScript{O}}
\def\Pp{\EuScript{P}}

\def\Tt{\EuScript{T}}

\def\Ga{\Gamma}
\def\La{\Lambda}

\def\a{\alpha} 
\def\b{\beta}
\def\d{\delta}
\def\e{\varepsilon}

\def\g{\gamma}

\def\k{\kappa}
\def\l{\lambda}

\def\p{\mathfrak{p}}
\def\s{\sigma}
\def\t{\theta}

\def\>{\geqslant}
\def\<{\leqslant}

\def\Hom{{\rm Hom}}

\def\Aut{{\rm Aut}}
\def\Mat{{\rm M}}
\def\GL{{\rm GL}}
\def\Gal{{\rm Gal}}

\def\tr{{\rm tr}}

\def\mult#1{{#1}^{\times}}

\def\qlb{{\overline{\mathbf{Q}}_\ell}}
\def\zlb{{\overline{\mathbf{Z}}_\ell}}
\def\flb{{\overline{\mathbf{F}}_{\ell}}}

\def\ip{\boldsymbol{i}}

\def\r{{\textbf{\textsf{r}}}}

\def\kk{\boldsymbol{k}}
\def\dd{\boldsymbol{d}}
\def\ee{\boldsymbol{e}}
\def\ll{\boldsymbol{l}}

\def\Sp{{\rm Sp}}

\def\kt{\widetilde\k}
\def\bkt{\boldsymbol{\kt}}
\def\bk{\boldsymbol{\k}}

\def\lt{\widetilde\l}
\def\blt{\boldsymbol{\lt}}
\def\bl{\boldsymbol{\l}}

\def\rt{{\widetilde\rho}}
\def\st{\widetilde\s}
\def\bs{\boldsymbol{\s}}
\def\bst{\boldsymbol{\st}}

\def\BJ{\boldsymbol{\J}}
\def\TT{\boldsymbol{\Theta}}
\def\0{\boldsymbol{0}}

\def\RA{\EuScript{R}} 
\def\XA{{\rm Irr}}
\def\GA{\EuScript{G}}

\def\rl{{\textbf{\textsf{r}}}_\ell}
\def\il{{\textbf{\textsf{i}}}_\ell}

\def\tt{\widetilde{\tau}}
\def\pit{\widetilde{\pi}}

\def\TT{\boldsymbol{\Theta}}

\def\BLT{\widetilde{\textbf{\textsf{J}}}_\ell}

\def\BL{{\textbf{\textsf{J}}}_\ell}

\def\plt{\widetilde{\boldsymbol{\pi}}_\ell}

\def\XL{\Y_{\ell}}
\def\XLT{{\X}_{\ell}}

\def\aa{\mathfrak{a}}
\def\bb{\mathfrak{b}}

\def\tdt{\times\dots\times}
\def\odo{\otimes\dots\otimes}

\def\car{\a}
\def\u{u}
\def\xw{\xi_{{\rm w}}}
\def\xt{\xi_{{\rm t}}}
\def\Ger{\G^{{\rm ell}}_{{\rm reg}}}

\def\Psit{\Psi_{{\rm t}}}
\def\bltt{\bs} 
\def\XR{\X_{{\rm reg}}}

\def\Om{\Omega}

\def\dag{*}
\def\ffr#1{\smash{\mathop{\longrightarrow}\limits^{#1}}}

\long\def\MSC#1\EndMSC{\def\arg{#1}\ifx\arg\empty\relax\else
     {\par\narrower\noindent%
     2010 Mathematics Subject Classification: #1\par}\fi}

\long\def\KEY#1\EndKEY{\def\arg{#1}\ifx\arg\empty\relax\else
	{\par\narrower\noindent Keywords and Phrases: #1\par}\fi}

\title[Explicit local Jacquet--Langlands correspondence]
{Towards an explicit local Jacquet--Langlands correspondence 
beyond the cuspidal case 
}

\author{Vincent Sécherre}
\address{Université de Versailles St-Quentin-en-Yvelines\\
Laboratoire de Mathémati\-ques de Versailles\\
45 avenue des Etats-Unis\\
78035 Versailles cedex, France}
\email{vincent.secherre@math.uvsq.fr}

\author{Shaun Stevens}
\address{School of Mathematics, University of East Anglia, 
  Norwich NR4 7TJ, United Kingdom}
\email{Shaun.Stevens@uea.ac.uk}

\begin{abstract}
We show how the modular representation theory of inner forms of general 
linear groups over a non-Archimedean local field can be brought to bear on the
complex theory in a remar\-kable way. 
Let $\F$ be a non-Archimedean locally compact field of residue 
characteristic $p$, and let $\G$ be an inner form of the general linear 
group $\GL_n(\F)$, $n\>1$.
We consider the problem~of~describ\-ing explicitly the local Jacquet--Langlands 
correspondence $\pi\mapsto{}_{{\rm JL}}\pi$
between the complex discrete series represen\-tations of $\G$ 
and $\GL_n(\F)$, in terms of type theory.
We show that the congruence properties of the local Jacquet--Langlands 
correspondence exhibited by A.~M\'\i nguez and the first~named author 
give information about the explicit description of this correspondence. 
We prove that the problem~of the invariance~of~the~endo-class~by the 
Jacquet--Langlands correspondence can be reduced to the case where the 
representations $\pi$ and ${}_{{\rm JL}}\pi$ are both cuspidal with torsion 
number $1$. 
We also give~an~ex\-pli\-cit description of the~Jacquet--Langlands 
correspondence 
for all {essentially tame} discrete series~re\-presentations of $\G$,
up to an~un\-ramified twist, in terms of admissible pairs,
generalizing~previous~re\-sults by Bushnell and Henniart. 
In positive depth, our results are the first beyond the case 
where~$\pi$ and ${}_{{\rm JL}}\pi$ are both cuspidal. 
\end{abstract}

\begin{document} 

\maketitle

\MSC 
22E50 
\EndMSC
\KEY 
Jacquet--Lang\-lands correspondence, 
$\ell$-blocks, 
Congruences mod $\ell$,
Endo-class, Admissible pair, Rectifier.
\EndKEY

\thispagestyle{empty}


\section{Introduction}

\subsection{} 

Let $\F$ be a non-Archimedean locally compact field of residue characteristic $p$, 
let $\H$ be the~general linear group 
$\GL_n(\F)$, $n\>1$, and let $\G$ be an inner form of $\H$.
This is a group of the form $\GL_m(\D)$, where $m$ divides $n$ and $\D$ is a 
central division $\F$-algebra whose reduced degree is denoted $d$,~with 
$n=md$. 
Let $\Dd(\G,\CC)$ denote the set of all isomorphism classes of 
essentially square 
inte\-gra\-ble, irreducible complex smooth representations of $\G$. 
The local Jacquet--Langlands cor\-res\-pon\-dence 
\cite{JL,Rog,DKV,BaduJL}
is a bi\-jec\-tion
\begin{eqnarray*}
\Dd(\G,\CC) & \to & \Dd(\H,\CC) \\
\pi & \mapsto & {}_{{\rm JL}}\pi
\end{eqnarray*}
specified by a character relation on elliptic regular conjugacy classes. 
Bushnell and Henniart~have elaborated a vast programme aiming at giving 
an explicit description of this correspondence~\cite{HJL1,BHJL2,BHLTL3,BHJL3}.
The present article is a contribution to this programme.

We first have to explain what we mean by an explicit description of the 
Jacquet--Lang\-lands cor\-res\-pondence. 
Essentially square inte\-gra\-ble representations 
of $\G$ can be descri\-bed in terms~of para\-bolic induction.
Given such a representation $\pi$,
there are a unique integer $r$ dividing $m$ and a cuspidal 
irreducible representation $\rho$ of $\GL_{m/r}(\D)$,
unique up to isomorphism, such that
$\pi$ is isomorphic to the unique irreducible quotient of the para\-boli\-cally 
induced representation
\begin{equation*}
\label{parsrho}
\rho \times \rho\nu^{s(\rho)} \times\dots\times \rho\nu^{s(\rho)(r-1)}
\end{equation*}
where $\nu$ is the unramified character ``absolute value of the reduced 
norm'' and $s(\rho)$ is a positive integer dividing $d$, associated to 
$\rho$ in \cite{Tadic}. 
The essentially square inte\-gra\-ble repre\-sen\-tation
$\pi$~is en\-tirely characterized by the pair $(\rho, r)$; 
this goes back to Bernstein--Zelevinski \cite{Zel} 
when $\D$~is equal to $\F$, and Tadi\'c \cite{Tadic} in the general case
(see also Badulescu \cite{BaduJIMJ} when $\F$ has positive~char\-acteristic). 
In particular, we may write $s(\pi)=s(\rho)$. 
Similarly, associated with the Jacquet-Lang\-lands transfer ${}_{{\rm JL}}\pi$, 
there are an integer $u$ dividing $n$ and a cuspidal 
irreducible representation $\s$ of $\GL_{n/u}(\F)$.
The integers $r,u$ are related by the identity $u=rs(\pi)$.
It remains to understand how the cuspidal representations $\rho,\s$ are 
related.

Thanks to the theory of simple types, developed by Bushnell and Kutzko 
\cite{BK} for the general linear group $\GL_n(\F)$ 
and by Broussous \cite{BrD}
and the authors \cite{VS1,VS2,VS3,SeSt1} for its inner~forms,
the cuspidal re\-presentation $\rho$ 
is compactly induced from a compact mod centre, open subgroup. 
More precisely, there are an \emph{extended maximal simple type},
made of a compact mod centre sub\-group $\BJ$ of $\GL_{m/r}(\D)$ 
and an ir\-re\-ducible representation $\bl$ of $\BJ$, 
both constructed in a very~spe\-cific way,
such that the compact induction of $\bl$ to $\GL_{m/r}(\D)$ is 
irreducible and isomorphic to $\rho$.
Such a type is uniquely deter\-mined up to con\-ju\-gacy. 
Giving an explicit description of the local Jac\-quet--Lang\-lands 
correspon\-dence
will thus consist of describing the extended maximal 
simple~type associated~with~the~repre\-sen\-tation $\s$ 
in terms of that~of $\rho$.

This programme was first carried out for essentially square integrable 
representations~of~depth zero, by Silberger--Zink \cite{SZ1,SZ2} and 
Bushnell--Henniart \cite{BHlevel0}. Before explaining the~other~cases which 
have already been dealt with, we need to introduce two numerical invariants 
associated~to an essentially square integrable, irreducible representation of 
$\G$. 
Such a representation~$\pi$ has:~a \emph{tor\-sion number $t(\pi)$}, the
number of unramified characters $\chi$ of $\G$ such that the twisted
represen\-tation $\pi\chi$ is isomorphic to $\pi$; and a 
\emph{parametric degree $\delta(\pi)$}, 
defined in \cite{BHJL3} via the theory~of~sim\-ple types, 
which is a multiple of $t(\pi)$ and divides $n$. 
Both of these integers are 
invariant un\-der~the Jac\-quet--Langlands correspondence \cite{BHJL3}. It is
interesting to note that the invariance of the para\-me\-tric degree implies that 
$\delta(\pi)s(\pi)=n/r$. Consequently, the representation ${}_{{\rm JL}}\pi$ 
is cuspidal if and only if the parametric degree of $\pi$ is equal to $n$. 

In \cite{BHJL3}, Bushnell and Henniart treat the case where the cuspidal 
representation $\pi$ is \emph{essentially tame} (that is, $\d(\pi)/t(\pi)$ is 
prime to $p$) and of parametric degree $n$. In that case, they explicitly 
describe the Jacquet--Langlands correspondence by parametrizing the conjugacy 
classes of exten\-ded maximal simple types in $\G$ and $\H$ by objects called 
\emph{admissible pairs} \cite{Howe}. (We will see these objects in 
Section~\ref{SECTION9}.) 

In \cite{BHLTL3},  they also  treat the  case which  is in  some sense  at the
opposite extreme to the essentially tame case, where $n$ is of the form $p^k$, 
with $k\>1$ and $p\neq 2$, and where $\pi$ is a cuspidal~repre\-sen\-tation of 
$\D^\times$ which is \emph{maximal totally ramified} (that is, $\d(\pi)=n$ and 
$t(\pi)=1$). 

In \cite{IT}, Imai and Tsushima treat the case where $\pi$ is an epipelagic 
cuspidal representation of $\G$, that is, of depth $1/n$. 
Such representations are maximal totally ramified. 

With the exception of \cite{SZ1,SZ2} and \cite{BHlevel0}, these results all 
concern cases where~the represen\-ta\-tions $\pi$ and ${}_{{\rm JL}}\pi$ are both 
cuspidal, that is, when $\pi$ is of parametric degree $n$. 
In such ca\-ses, 
since the cuspidal representation $\pi$ can be expressed as the compact 
induction of an exten\-ded~maxi\-mal simple type $(\BJ,\bl)$, there is a 
relatively straightforward formula giving the trace~of~$\pi$ at an elliptic 
regular element in terms of the trace of $\bl$ (see 
\cite[Theorem~A.14]{BHLTL1} and \cite[(1.2.2)]{BHJL3}). The strategies 
followed in \cite{BHJL3,BHLTL3} and \cite{IT} depend crucially on such a 
formula. 
When considering~a non-cuspidal essentially square integrable 
representation, we are in a much less favourable situation. 
For the group $\GL_n(\F)$, 
Broussous \cite{BrCJM} and Broussous--Schneider \cite{BrSch} have 
ob\-tained for\-mulae expressing the trace of such a representation at an elliptic 
regular element by bringing in the theo\-ry of simple types. 
However, in this 
article, we follow a different route. 

\subsection{}

An important first step towards the general case is to look at the behavior of 
the local Jacquet--Lang\-lands correspondence with respect to 
\emph{endo-classes}. 
An endo-class is a type-theoretic inva\-riant as\-so\-ciated to any 
essentially square inte\-gra\-ble representation of $\G$, whose 
construction~re\-quires a considerable machinery \cite{BHLTL1,BSS}. 
However, for cuspidal representations of $\H$, 
it turns out to have a rather simple arithmetical interpretation through the 
local Langlands correspondence~\cite{BHLTL4}. In\-deed, 
two cuspidal irreducible representations of general linear groups over $\F$ 
have the same endo-class if and only if the irreducible representations 
of the absolute Weil group $\EuScript{W}_\F$ associated 
to them by the local Langlands cor\-res\-pon\-dence share an irreducible 
component when restricted to the wild inertia subgroup $\EuScript{P}_\F$.
The local Langlands correspondence thus induces a bijection between the set 
of $\EuScript{W}_\F$-conjugacy classes of irreducible representations of 
$\EuScript{P}_\F$ and the set $\Ee(\F)$ of endo-classes over~$\F$.

It is expected that the local Jacquet--Lang\-lands correspondence 
preserves endo-classes. 
More precisely, there is the following conjecture.

\medskip

\noindent\textbf{\textit{Endo-class Invariance Conjecture}.}
\textit{
For any essentially square inte\-gra\-ble, irreducible complex representation\/ 
$\pi$ of\/ $\G$, the endo-classes of\/ $\pi$ and\/ ${}_{{\rm JL}}\pi$ are the same.}

\medskip

Our first main result is the following (see Theorem~\ref{MainTheoremA}), 
which reduces this conjecture to the case of maximal totally ramified cuspidal 
representations. 

\medskip

\noindent\textbf{\textit{Theorem A}.}
\textit{
Assume that, for all\/ $\F$ and\/ $n$, and all 
cusp\-idal irreducible complex representations\/ $\pi$ of\/ $\G$ 
such that\/ $\d(\pi)=n$ and\/ $t(\pi)=1$, 
the cuspidal representations\/ $\pi$ and\/ ${}_{{\rm JL}}\pi$ 
have the same endo-class. 
Then the Endo-class Invariance Conjecture is true.}

\medskip

Before explaining our strategy, we must first make a detour through the 
modular representa\-tion theory of $\G$ and explain recent developments 
concerning the modular properties of the~Jac\-quet--Langlands correspondence. 
Fix a prime number $\ell$ different from $p$, and consider the smooth 
\emph{$\ell$-adic} representations of $\G$, that is, with coefficients in the 
algebraic closure $\qlb$ of the field of $\ell$-adic numbers. 
There is then 
the notion of \emph{integral} irreducible representation of $\G$: containing a
$\G$-stable $\zlb$-lattice (where $\zlb$ is the ring of integers of $\qlb$), 
which can then be reduced modulo~$\ell$. 
More precisely, given such a
representation $\pi$ containing a stable $\zlb$-lattice $\La$, Vign\'eras
\cite{Vigb,Vigw}  showed that  the  representation $\La\otimes_{\zlb}\flb$  is
smooth of finite length (where $\flb$ is the residue field of $\zlb$), and its
semisimplification  is  independent of  the  choice  of  $\La$; we  call  this
semisimplification the \emph{reduction  mod~$\ell$} of $\pi$. Thus  we can say
that two integral irreducible $\ell$-adic representations of $\G$ are 
\emph{congruent mod~$\ell$} if their reductions mod~$\ell$ are isomorphic. 

To relate this to the local Jacquet--Langlands correspondence, we fix an 
isomorphism of fields between $\CC$ and $\qlb$; replacing one by the other via 
this isomorphism, we get an $\ell$-adic Jacquet--Langlands correspondence 
\begin{equation*}
\Dd(\G,\qlb)\ \ffr{\simeq}\ \Dd(\H,\qlb)
\end{equation*}
which is independent of the choice of isomorphism. Thus one can study the 
compatibility of this correspondence with the relation of congruence mod 
$\ell$, which was done by Dat~\cite{Datj} and then in full generality by M\'\i 
nguez and the first author~\cite{MSj}: two integral representations of 
$\Dd(\G,\qlb)$ are congruent mod~$\ell$ if and only their images under the 
$\ell$-adic Jacquet--Langlands correspondence are congruent mod~$\ell$ 
(\cite[Th\'eor\`eme 1.1]{MSj}). 

We now need to explain how modular representation theory can give us 
information on the com\-plex representation theory. 
The starting point for our 
strategy to prove Theorem~A using~mo\-dular methods is the fact that two 
representations of $\Dd(\G,\qlb)$ which are congruent mod~$\ell$ have the same 
endo-class. The converse is, of course, not true but we will see that one can 
neverthe\-less link two essentially square integrable representations with the 
same endo-class by a chain~of congruence relations. 
Let us explain this in more detail. 

Firstly, for any irreducible $\ell$-adic representation of $\G$, we have a 
notion of \emph{mod-$\ell$ inertial~cuspi\-dal support} 
(Definition~\ref{defmodliss}, and also~\cite{Helm} in the split case), coming 
from the notion of supercuspidal support for irreducible representations of 
$\G$ with coefficients in $\flb$, defined in~\cite{MSc}. 
Two ir\-reducible 
complex representations of $\G$ are said to be \emph{$\ell$-linked} 
(Definitions~\ref{defllinked} and~\ref{lblock}) if there is a field 
isomorphism $\CC\simeq\qlb$ such that the resulting irreducible $\ell$-adic 
representations have the same mod-$\ell$ inertial cuspidal support. This is 
independent of the choice of field isomorphism and it is not hard, using the 
work done in~\cite{MSj}, to show that the Jacquet--Langlands correspondence 
preserves the relation of being $\ell$-linked for essentially 
square-integrable representations (Pro\-position~\ref{lLINKEDJL} and 
Corollary~\ref{lLINKEDJLcoro}). We can now introduce the following definition 
(Definition \ref{deflinked}). 

\medskip

\noindent\textbf{\textit{Definition}.}
\textit{Two irreducible complex representations\/ $\pi,\pi'$ of\/ $\G$ are 
  said to be \emph{linked} if there are a finite sequence of prime numbers\/ 
  $\ell_1,\ldots,\ell_r$, all different from\/ $p$, and a finite sequence of 
  ir\-reducible complex representations\/ $\pi=\pi_0,\pi_1,\ldots,\pi_r=\pi'$ 
  such that, for each\/ $i\in\{1,\dots,r\}$, the representations\/ $\pi_{i-1}$ 
  and\/ $\pi_i$ are\/ $\ell_i$-linked.} 

\medskip

Two  essentially square  integrable complex  representations which  are linked
have the same  endo-class. More generally, if we  define the \emph{semi-simple
  endo-class} of an  irreducible representation to be the  weighted formal sum
of the  endo-classes of the  cuspidal representations in its  cuspidal support
(with   multiplicities   determined   by   the  sizes   of   the   groups   --
see~\eqref{SSEC}), then two irreducible  representations which are linked have
the same semi-simple endo-class. 
The interest of the definition is apparent 
from the following theorem (see Theorem~\ref{MAINTHEOREM11}), which says that 
the converse is also true. 

\medskip

\noindent\textbf{\textit{Theorem B}.}
\textit{Two irreducible complex representations of\/ $\G$ are linked if and 
  only if they have the same semi-simple endo-class.} 

\medskip

In particular, two essentially square integrable complex representations have 
the same endo-class if and only if they are linked; moreover, one can then 
link them by a sequence of essentially square integrable representations 
(Remark~\ref{Solgrub}). 

Theorem~B gives a remarkable reinterpretation of what it means for two 
irreducible complex representations to have the same semi-simple endo-class. 
Beyond the intrinsic interest in explica\-ting the notion of endo-class and its 
relation with modular representation theory, the main in\-te\-rest in this 
reformulation comes from the fact that, applying results from~\cite{MSj}, we 
are able to prove the following (Theorem~\ref{TheoremC}). 

\medskip

\noindent\textbf{\textit{Theorem C}.}
\textit{Two essentially square integrable complex representations of\/ $\G$ 
  are linked if and only if their transfers to\/ $\H$ are linked.} 

\medskip

It follows from Theorems~B and~C that two essentially square integrable 
complex represen\-tations of $\G$ have the same endo-class if and only if their 
transfers to $\H$ have the same endo-class. Thus, denoting by $\Ee_n(\F)$ the 
set of endo-classes over $\F$ of degree dividing $n$, the Jacquet--Lang\-lands 
correspondence induces a bijection 
\begin{equation*}
\boldsymbol{\pi}_1 : \Ee_n(\F) \to  \Ee_n(\F).
\end{equation*}
We now observe the following fact (Proposition~\ref{propintro}). 

\medskip

\noindent\textbf{\textit{Proposition}.}
\textit{For every essentially square integrable complex representation of\/ 
  $\G$, there is a cuspidal complex representation of\/ $\G$ with the same 
  endo-class and with parametric degree\/ $n$.} 

\medskip

To prove the conjecture -- that is, to prove that $\boldsymbol{\pi}_1$ is the 
identity map -- it is therefore~sufficient to prove that, for every cuspidal 
complex representation $\pi$ of $\G$ of parametric degree $n$, 
the~re\-pre\-sentations $\pi$ and ${}_{{\rm JL}}\pi$ have the same endo-class. Using 
techniques developed in~\cite[Sec\-tion~6]{BHJL3}, we can go further and 
show that one need only consider cuspidal representations of parametric degree 
$n$ and torsion number $1$, thus obtaining Theorem~A. Therefore, to prove the 
Endo-class Invariance Conjecture, it remains only to prove the following 
conjecture. 
Say that an endo-class is \textit{totally ramified} if it has 
residual degree $1$, 
that is, if its tame parameter field 
(in the sense of \cite[Section 2]{BHToAnEffectiveLC}) is totally ramified.

\medskip

\noindent
{\textbf{\textit{Conjecture}.}
\textit{For all\/ $\F$ and $n$, and for every totally ramified 
  $\F$-endo-class~$\TT$ 
of degree $n$, there is a cuspidal complex representation $\pi$ of $\G$
with endo-class $\TT$ such that ${}_{{\rm JL}}\pi$ has endo-class $\TT$.}
}

\subsection{}

We now leave to one side the preservation of endo-classes and pass to the next 
step towards and explicit description of the Jacquet--Langlands 
correspondence. 
We will see that the modu\-lar methods described in the previous
paragraphs can be pushed further to yield additional~in\-formation. 
Let $\TT$ be an endo-class of degree dividing $n$ and suppose it 
is invariant under~the
Jacquet--Langlands correspondence, i.e.~$\boldsymbol{\pi}_1(\TT)=\TT$. 
(This is the case, for example, if $\TT$ is~ess\-entially tame or epipelagic.) 
The correspondence thus induces a bijection between iso\-mor\-phism classes of 
essentially square integrable complex representations of $\G$ with endo-class 
$\TT$, and those of $\H$. Since the correspondence is also compatible with 
unramified twisting, we get~a~bijec\-tion 
\begin{equation*}
\Dd_0(\G,\TT)\ \ffr{\simeq}\ \Dd_0(\H,\TT),
\end{equation*}
where $\Dd_0(\G,\TT)$ denotes the set of inertia classes of essentially square 
integrable complex repre\-sen\-tations of $\G$ with endo-class $\TT$. The theory 
of simple types~\cite{BK,VS3,SeSt1,SeSt2} gives us~a~cano\-nical bijection between 
$\Dd_0(\G,\TT)$ and the set $\Tt(\G,\TT)$ of $\G$-conjugacy classes of simple 
types for $\G$ with endo-class $\TT$. More precisely, the inertia class of an 
essentially square integrable~com\-plex representation $\pi$ corresponds to the 
conjugacy class of a simple type $(\J,\l)$, formed of a compact open subgroup 
$\J$ of $\G$ and an irreducible representation $\l$ of $\J$, if and only if 
$\l$ is an irreducible component of the restriction of $\pi$ to $\J$. Thus we 
get a bijection 
\begin{equation}
\label{JLTS}
\Tt(\G,\TT)\ \ffr{\simeq}\ \Tt(\H,\TT).
\end{equation}
To go further, we need to enter into the detail of the structure of simple 
types (Paragraph~\ref{Par5}). 

Given a simple type $(\J,\l)$ of $\G$ with endo-class $\TT$, the group $\J$ 
has a unique normal pro-$p$ subgroup, denoted~$\J^1$. The restriction of $\l$ 
to $\J^1$ is isotypic, that is, it is a direct sum of copies of a single 
irreducible representation $\eta$. This representation $\eta$ can be extended 
to a representation of $\J$ with the same intertwining set as $\eta$. If we 
fix such an extension $\k$, then the representation $\l$ can be expressed in 
the form $\k\otimes\s$, where $\s$ is an irreducible representation of $\J$, 
trivial on $\J^1$. Moreover, the quotient group $\J/\J^1$ is (non-canonically) 
isomorphic to a product of copies of a single finite general linear groups and 
$\s$, viewed as a representation of such a product of groups, is the tensor 
product of copies of a single cuspidal representation. A theorem of 
Green~\cite{Green} allows us to parametrize $\s$ by a character of 
$\kk^\times$, where $\kk$ is a certain finite extension of the residue field 
of $\F$, which depends only on the endo-class $\TT$ and on $n$. 
This character 
is determined up~to conjugation by the Galois group $\Gal(\kk/\ee)$ of $\kk$ 
over a subfield $\ee$ which depends only on $\TT$. 

We denote by $\X$ the group of characters of $\kk^\times$ and by $\Ga$ the 
Galois group $\Gal(\kk/\ee)$. Fixing once and for all a choice of 
representation $\k$ for a \textit{maximal} 
simple type in $\G$ with endo-class~$\TT$, 
we get a bijection from $\Ga\backslash\X$ to $\Tt(\G,\TT)$ (see 
Paragraph~\ref{Par5} for details). 
Making a similar choi\-ce for the group $\H$, 
we also obtain a bijection from $\Ga\backslash\X$ to $\Tt(\H,\TT)$. 
Composing 
with the~bi\-jec\-tion~\eqref{JLTS}, we thus obtain a permutation 
\begin{equation*}
\Upsilon : \Ga\backslash\X \to \Ga\backslash\X
\end{equation*}
which depends on the choice of $\k$ and of its analogue for $\H$. 

We write $[\car]$ for the $\Ga$-orbit of a character $\car\in\X$. 
The following result 
(see Proposition~\ref{ExpliUpsi}), which again is proved via modular methods, 
suggests that, in order to determine the permutation $\Upsilon$ it is 
sufficient to compute the value of $\Upsilon([\car])$ for certain characters 
$\car$ only. 

\medskip

\noindent\textbf{\textit{Proposition}.}
\textit{
Let $\car\in\X$ and let $\ll$ be the unique subfield of $\kk$ such that the 
stabilizer of $\car$ in $\Ga$ is $\Gal(\kk/\ll)$. Suppose there are an 
$\ee$-regular character $\b\in\X$ and a prime number $\ell\neq p$ prime to 
the order of $\ll^\times$ such that the order of $\b\car^{-1}$ is a power of 
$\ell$. Suppose further that $\Upsilon([\b])=[\b\mu]$, for some character 
$\mu\in\X$. Then $\Upsilon([\car])=[\car\nu]$, where $\nu\in\X$ is the unique 
character of order prime to $\ell$ 
such that $\mu\nu^{-1}$ has order a power of $\ell$. } 

\medskip

In fact we need a more powerful version of this result, which we do not 
explain here, which~re\-quires being able to pass from $\G$ to a bigger group 
$\GL_{m'}(\D)$, with $m'>m$. 
(See Section~\ref{Sec8},~in~par\-ti\-cular Paragraph~\ref{nextss}.) 

To conclude, in the final section of the paper, we illustrate this principle 
in the essentially tame case where, for an appropriate choice of $\k$ and of 
its analogue for $\H$, the value of $\Upsilon([\car])$ is known for all 
$\ee$-regular characters $\car\in\X$ (see \cite{BHJL3}). We end with the following 
result (Theorem~\ref{MT9}). 

\medskip

\noindent\textbf{\textit{Theorem D}.}
\textit{
Let $\TT$ be an essentially tame $\F$-endo-class of degree dividing $n$. 
There~is~a~cano\-nically determined character $\mu$ of $\kk^\times$, depending 
only on $m$, $d$ and $\TT$, such that $\mu^2=1$ and 
\begin{equation*}
\Upsilon([\car]) = [\car\mu],
\end{equation*}
for all characters $\car\in\X$. More precisely, $\mu$ is the 
``rectifier'' given by Bushnell--Henniart's First Compari\-son Theorem 
\cite{BHJL3} {\rm 6.1}. 
}

\medskip

See also Theorem \ref{THPAI}, which is a reformulation of Theorem D 
in terms of admissible pairs.

\section{Notation}
\label{Notation}

We fix a non-Archimedean locally compact field $\F$ with residual 
characteristic $p$.

Given $\D$ a finite dimensional central division $\F$-algebra and a 
positive integer $m\>1$, we write $\Mat_{m}(\D)$ for the algebra 
of $m\times m$ matrices with coefficients in $\D$ and $\GL_{m}(\D)$ 
for the group of~its invertible elements. 
Choose an $m\>1$ and write $\G=\GL_m(\D)$.

Given an algebraically closed field $\R$ 
of characteristic different from $p$, we will consider smooth representations of 
the locally profinite group $\G$ with coefficients in $\R$. 
We write $\XA(\G,\R)$ for the set of isomorphism classes of irreducible 
represen\-ta\-tions of $\G$ and $\RA(\G,\R)$ for the Gro\-then\-dieck 
group of its finite length represen\-ta\-tions, identified with the free 
abelian group with basis $\XA(\G,\R)$.
If $\pi$ is a representation of $\G$, the integer $m$ is called its 
\emph{degree}. 

Given $\a=(m_{1},\ldots,m_{r})$ a family of positive integers of sum $m$, 
we write $\ip_\a$ for the functor of standard parabolic induction associated 
with $\a$, normalized 
with respect to the choice of a square root in the field $\R$
of the cardinality $q$ of the residual field of $\F$.
Given, for each $i\in\{1,\ldots,r\}$, a re\-presenta\-tion $\pi_{i}$ of 
$\GL_{m_i}(\D)$, we write
\begin{equation*}
\pi_1\times\cdots\times\pi_r=\ip_{\a}(\pi_1\otimes\cdots\otimes\pi_r).
\end{equation*}

Given a representation $\pi$ and a character $\chi$ of $\G$, 
we write $\pi\chi$ for the twisted representation defined by 
$g\mapsto\chi(g)\pi(g)$.

We fix once and for all 
a smooth additive 
character $\psi:\F\to\mult\R$, trivial on the maximal ideal $\mathfrak{p}$ 
of the ring of integers $\Oo$ of $\F$ but not trivial on $\Oo$. 

\section{Preliminaries}

In this section we let $\R$ be an algebraically closed field of characteristic 
different from $p$.
Write $\G=\GL_m(\D)$ for some positive integer $m\>1$.
Write $d$ for the reduced degree of $\D$ over $\F$, and define $n=md$.

\subsection{}
\label{P21}

Let $\rho$ be a cuspidal irreducible $\R$-representation of $\G$. 
Associated with $\rho$ there is an unramified character $\nu_{\rho}$ of $\G$ 
(see \cite[Section~4]{MSt}).
In \cite{MSc} we attach to $\rho$ and any integer $r\>1$ 
an irredu\-ci\-ble subrepresentation $\Z(\rho,r)$ and an irreducible quotient 
$\L(\rho,r)$ of the induced representation 
\begin{equation}
\label{INDRHO}
\rho\times\rho\nu_{\rho}^{}\tdt\rho\nu_{\rho}^{r-1}
\end{equation}
(see \cite[Paragraph 7.2 and Définition 7.5]{MSc}).

When $\R$ is the field of complex numbers, $\Z(\rho,r)$ and $\L(\rho,r)$ 
are uniquely determined in this way, and all 
essentially square integrable 
representations of $\G$ are 
isomorphic to a representation 
of the form $\L(\rho,r)$ for a unique pair $(\rho,r)$. 

For an arbitrary $\R$, the representation $\L(\rho,r)$ is called a 
\textit{discrete series} $\R$-representation
of $\G$ and $\Z(\rho,r)$ is called a \textit{Speh} $\R$-representation.
If $\rho$ is super\-cuspidal, $\Z(\rho,r)$
is called a super-Speh representation.

Associated to $\rho$ there is also its 
parametric degree $\delta(\rho)$ (see \cite[Section~2]{BHJL3}).
The parametric degree 
is related to the integer $s(\rho)$ defined in \cite[Paragraph~3.4]{MSt} 
by the formula $s(\rho)\delta(\rho)=n$, 
and the character $\nu_\rho$ is equal to $\nu^{s(\rho)}$.

According to \cite[Paragraph~8.1]{MSc},
where the notion of residually nondegenerate 
representa\-tion is defined, the induced representation \eqref{INDRHO} 
contains a unique residually nondegene\-rate~irreduci\-ble subquotient, 
denoted
\begin{equation*}
\Sp(\rho,r).
\end{equation*}
When $\R$ has characteristic $0$, this is equal to $\L(\rho,r)$. 
When $\R$ has characteristic $\ell>0$ however, it may differ from
$\L(\rho,r)$ (see \cite[Remark~8.14]{MSc}).

Assume $\R$ has characteristic $\ell>0$, 
and let us write $\omega(\rho)$ for the smallest positive integer $i\>1$ 
such that 
$\rho\nu_\rho^i$ is isomorphic to $\rho$.
Then the irreducible representation
\begin{equation}
\label{SPCUSP}
\Sp(\rho,\omega(\rho)\ell^v)
\end{equation}
is cuspidal for any integer $v\>0$.
Moreover, any cuspidal non-supercuspidal irreducible represen\-tation is of 
the form \eqref{SPCUSP} for a \emph{supercuspidal} irreducible representation 
$\rho$ and a unique integer $v\>0$ (see \cite[Théorème~6.14]{MSc}).
We record this latter fact for future reference. 

\begin{prop}
\label{Coupure}
Assume $\R$ has positive characteristic $\ell$,
and let $\rho$ be a cuspidal irreducible representation of $\G$. 
There are a unique positive integer $k=k(\rho)$ and a supercuspidal irreducible 
representation $\tau$ of degree $m/k$ such that $\rho$ is isomorphic to 
$\Sp(\tau,k)$.
\end{prop}

\subsection{}

In this paragraph, we assume that $\R$ is an algebraic closure $\qlb$ 
of the field of $\ell$-adic numbers.
Let $\rt$ be an $\ell$-adic cuspidal irreducible representation of $\G$. 
Assume $\rt$ is integral \cite{Vigb} and write $a=a(\rt)$ for the 
length of its reduction mod $\ell$, denoted $\rl(\rt)$.

\begin{prop}[{\cite[Theorem~3.15]{MSt}}]
\label{redcusp}
Let $\rho$ be an irreducible factor of $\rl(\rt)$.
Then
\begin{equation*}
\rl(\rt) = \rho+\rho\nu+\dots+\rho\nu^{a-1},
\end{equation*}
where $\nu$ denotes the unramified mod $\ell$ character ``absolute value of the
reduced norm''. 
\end{prop}

Now write $k(\rho)$ for the positive integer associated with $\rho$ 
by Proposition \ref{Coupure}.
Let us write
\begin{equation*}
w(\rt)=a(\rt)k(\rho).
\end{equation*}
This integer has proved to have important properties 
with respect to the local Jacquet--Langlands correspondence \cite{MSj}.

\subsection{}
\label{Par5}

We assume the reader is familiar with the language of simple types. 
For a detailed presentation, see \cite{VS3,MSt}.
For simple strata, we will use the simplified notation of 
\cite[Chapter~2]{BHToAnEffectiveLC}. 

Let $[\aa,\b]$ be a simple stratum in 
the simple central $\F$-algebra $\Mat_{m}(\D)$.
The centralizer of $\b$ in it, denoted $\B$, 
is a simple central $\F[\b]$-algebra.
There are an integer $m'\>1$ and an $\F[\b]$-division algebra 
$\D'$ such that
\begin{equation}
\label{ISOB}
\B \simeq \Mat_{m'}(\D').
\end{equation}
Assume that $\bb=\aa\cap\B$ is a maximal order in $\B$, 
and let us fix an isomorphism \eqref{ISOB} such that the image 
of $\bb$ is the maximal order made of all matrices with integer 
entries. 

Let $(\J,\k)$ be a $\b$-extension with coefficients in $\R$ associated 
with the simple stratum $[\aa,\b]$, that is, we have $\J=\J(\aa,\b)$ 
and the restriction of $\k$ to $\J^1=\J^1(\aa,\b)$ is the Heisenberg 
representation of a simple $\R$-character  associated with $[\aa,\b]$.
Moreover, any $y\in\mult\B$ intertwines $\k$.
(A $\b$-exten\-sion associated with such a simple stratum, that is, such that 
$\bb=\aa\cap\B$ is maximal in $\B$, is said to be \emph{maximal}.)
We write $\TT$ for the endo-class defined by this simple character 
\cite{BHLTL1,BSS}. 
Thanks to the isomorphism \eqref{ISOB} we identify the two groups
\begin{equation*}
\J/\J^1 \simeq \GL_{m'}(\dd),
\end{equation*}
where $\dd$ is the residue field of $\D'$.
We write $\GA$ for the group on the right hand side. 

Recall \cite{VS1,MSt} that the set $\Cc(\aa,\b)$ of simple characters 
associated with $[\aa,\b]$ depends on the choice of 
the smooth additive character $\psi$ fixed in Section \ref{Notation}.

We fix a finite extension $\kk$ of $\dd$ of degree $m'$.
We write $\Sigma$ for the Galois group of this extension 
and $\X$ for the group of $\R$-characters of $\mult\kk$.
Given $\car\in\X$, there is a unique subfield 
$\dd\subseteq\dd[\car]\subseteq\kk$ such that 
the $\Sigma$-stabilizer of $\car$ is $\Gal(\kk/\dd[\car])$ and then 
a character $\car_0$ of $\dd[\car]^\times$ 
such that $\car$ is equal to 
$\car_0$ composed with the norm of $\kk$ over $\dd[\car]$.
If we write $\u$ for the degree of $\dd[\car]/\dd$,
then $\car_0$ 
defines a supercuspidal irre\-ducible $\R$-re\-pre\-sentation 
$\s_0$ of $\GL_{\u}(\dd)$ --- see \cite{Green} if $\R$ has
characteristic $0$ and \cite{Dipper} or \cite{MSf} otherwise.

\begin{rema}
\label{PetrusMartel}
More precisely, if $\R$ has characteristic $0$, 
fix an embedding of $\dd[\car]$ in $\Mat_{\u}(\dd)$.
Then $\s_0$ is the unique 
(up to isomorphism) irreducible re\-pre\-sentation 
of $\GL_{\u}(\dd)$ such that
\begin{equation*}
\tr\ \s_0 (g) = (-1)^{\u-1}\cdot\sum\limits_{\g} \a_0^\g(g),
\end{equation*}
for all $g\in\dd[\car]^\times$ of degree $\u$ over $\dd$, 
where $\g$ runs over $\Gal(\dd[\car]/\dd)$.
\end{rema}

The character $\car\in\X$ thus defines a supercuspidal $\R$-representation
\begin{equation*}
\bs(\car)=\s_0\odo\s_0
\end{equation*}
of the Levi subgroup 
$\GL_{\u}(\dd)\tdt\GL_{\u}(\dd)$ in $\GA$.
Moreover, the fibers of the map $\car\mapsto\bs(\car)$
are the $\Sigma$-orbits of $\X$.
Now write $r$ for the integer defined by $ru=m'$.
The maximal order $\bb$ contains a unique principal order $\bb_r$ 
of period $r$, whose image by \eqref{ISOB} is standard. 
Write $\aa_r$ for the unique order normalized by $\F[\b]^{\times}$ 
such that $\aa_{r}\cap\B=\bb_{r}$,
and $\k_r$ for the transfer of~$\k$ with respect to the simple 
stratum $[\aa_r,\b]$ in the sense of \cite[Proposition~2.3]{MSt}. 
Considering $\bs(\car)$ as a representation of the group 
$\J_r=\J(\aa_r,\b)$ trivial on $\J^1(\aa_r,\b)$, 
we define
\begin{equation*}
\bl(\car) = \k_r\otimes\bs(\car),
\end{equation*}
which is a simple supertype in $\G$ defined on $\J_r$ 
in the sense of \cite{SSb}.
Write $\Ga$ for the Galois group of $\kk$ over $\ee$, where $\ee$ denotes 
the residue field of $\F[\b]$.

\begin{prop}
\label{LaSURJ}
The map
\begin{equation}
\label{MAPBL}
\car \mapsto \bl(\car) 
\end{equation}
induces a surjection from $\X$ 
onto the set of $\G$-conjugacy classes 
of simple supertypes in $\G$ with endo-class $\TT$.
Its fibers are the $\Ga$-orbits of $\X$.
\end{prop}

\begin{proof}
Surjectivity follows from the definition of simple supertype 
(\cite[Paragraph~2.2]{SSb})
and from the fact that any super\-cuspidal irreducible $\R$-representation of 
$\GA$ is of the form $\bs(\car)$ for some $\car\in\X$ with 
trivial $\Sigma$-stabilizer. 

The description of the fibers follows from \cite[Theorem~7.2]{SeSt2} 
together with the fact that the map 
$\car\mapsto\bs(\car)$ is $\Ga$-equivariant, with fibers 
the $\Sigma$-orbits of $\X$.
\end{proof}

Write $\Tt(\TT,\R)$ for the set of isomorphism classes of simple 
$\R$-supertypes of endo-class $\TT$.

\begin{prop}
The bijection
\begin{equation}
\label{BIJfk}
\{\text{$\Ga$-orbits of $\X$}\}
\leftrightarrow
\{\text{$\G$-conjugacy classes of $\Tt(\TT,\R)$}\}
\end{equation}
depends only on the choice of $\k$, not on that of the isomorphism 
\eqref{ISOB}. 
\end{prop}

\begin{proof}
Choosing another isomorphism $\B\simeq\Mat_{m'}(\D')$
such that the image of $\bb$ is the maximal order made of all matrices 
with integer entries has the effect 
-- according to the Skolem--Noether theorem -- 
of conjugating by an element $g\in\GL_{m'}(\D')$ normalizing this standard 
maximal order.

Consequently, if $\bs'(\car)$ is the representation of $\J_r$ trivial on 
$\J^1(\aa_r,\b)$ corresponding to $\car$ with respect to that choice of 
isomorphism, it differs from $\bs(\car)$ by conjugating by $g$.
\end{proof}

\subsection{}

We call an inertial class of supercuspidal pairs of $\G$ 
\emph{simple} if it contains a pair of the form
\begin{equation}
\label{nadir}
(\GL_{m/r}(\D)^r,\rho\otimes\dots\otimes\rho)
\end{equation}
for some integer $r$ 
dividing $m$ and some supercuspidal $\R$-repre\-sentation $\rho$ of 
$\GL_{m/r}(\D)$, and we 
define the endo-class of such an inertial class to be the endo-class 
of $\rho$. 
By \cite[Section~8]{SSb}, 
there is a bijective correspondence between simple inertial classes of 
supercuspidal pairs of the group $\G$ and $\G$-con\-jugacy 
classes~of simple supertypes of $\G$, that preserves 
endo-classes.~More~precise\-ly, the inertial class of \eqref{nadir}, denoted 
$\Omega$, 
corresponds to the 
$\G$-con\-jugacy class of a simple super\-type $(\J,\l)$ if and only if the 
irreducible representations of $\G$ 
occurring as a subquotient of the compact induction of $\l$ to $\G$
are exact\-ly those irreducible representations of $\G$ 
occurring as~a sub\-quotient~of the 
parabolic induction to $\G$ of an element of $\Omega$.

From the previous paragraph, we have an endo-class $\TT$ 
and a maximal $\b$-extension $\k$. 
Combi\-ning the map \eqref{MAPBL} with the correspondence between simple inertial 
classes of supercuspidal pairs and conjugacy classes of simple supertypes, 
we get a surjective map 
\begin{equation}
\label{PARISSC}
\car \mapsto \boldsymbol{\Omega}(\car)
\end{equation}
from $\X$ onto the set of simple inertial classes of supercuspidal pairs of 
$\G$ with associated endo-class $\TT$.
Its fibers are the $\Ga$-orbits of $\X$.

Let us recall the following important result from \cite[Théorème~8.16]{MSc}: 
given an irreducible repre\-sen\-tation $\pi$ of $\G$,
there are integers $m_1,\dots,m_r\>1$ such that $m_1+\dots+m_r=m$, 
and supercuspidal irreducible 
representations $\rho_1,\dots,\rho_r$ of $\GL_{m_1}(\D),\dots,\GL_{m_r}(\D)$ 
respectively, such that $\pi$ occurs as a subquotient of the induced 
representation $\rho_1\times\dots\times\rho_r$.
Moreover, up to renumbering, 
the supercuspidal representations $\rho_1,\dots,\rho_r$ are unique. 
The conjugacy class of the supercuspidal pair 
$(\GL_{m_1}(\D)\times\dots\times\GL_{m_r}(\D),\rho_1\otimes\dots\otimes\rho_r)$ 
is called the \emph{supercuspidal support} of $\pi$.

Let us call an irreducible $\R$-representation of $\G$ \emph{simple} if the
inertial class of its supercuspidal support is simple. 
For instance, any discrete series $\R$-representation of $\G$ is simple.

\begin{defi}
\label{defparaclass}
Let $\pi$ be a simple irreducible representation of $\G$ with endo-class $\TT$. 
The \emph{para\-metrizing class} of $\pi$ is the $\Ga$-orbit of a 
character $\car\in\X$ such that the two following equivalent conditions hold:
\begin{enumerate}
\item 
the supercuspidal support of $\pi$ belongs to the inertial class 
$\boldsymbol{\Omega}(\car)$; 
\item
the representation $\pi$ occurs as a subquotient of the compact 
induction of $\bl(\car)$ to~$\G$.
\end{enumerate}
The para\-metrizing class of $\pi$ is denoted $\X(\k,\pi)$, 
or simply $\X(\pi)$ if there is no ambiguity on the maximal $\b$-extension $\k$.
\end{defi}

\begin{rema}
\label{depkappa}
Let $\k'$ be another maximal $\b$-extension in $\G$.
By \cite[Théorème~2.28]{VS2} there is a character $\chi$
of $\ee^\times$ such that $\k'=\k\xi$, where $\xi$ is 
the character of $\J$ trivial on $\J^1$ that corresponds to 
the character $\chi\circ\N_{\dd/\ee}\circ\det$ of $\GA$,
where $\N_{\dd/\ee}$ is the norm map with respect to $\dd/\ee$.
Then we have $\a'\in\X(\k',\pi)$ if and only if 
$\a'\mu\in\X(\k,\pi)$, where $\mu$ is the character $\chi\circ\N_{\kk/\ee}$ of 
$\kk^\times$. 
\end{rema}

\begin{rema}
When $\R$ has characteristic $0$, the two equivalent conditions of Definition 
\ref{defparaclass} are also equivalent to:
\begin{enumerate}
\item[(3)]
the representation $\pi$ occurs as a quotient of the compact 
induction of $\bl(\car)$ to~$\G$.
\end{enumerate}
Equivalently, the restriction of $\pi$ to $\J_r$ contains $\bl(\car)$ as a 
subrepresentation.
\end{rema}

\section{Linked $\ell$-adic representations}

In this section, we fix a prime number $\ell$ different from $p$
and write $\G=\GL_m(\D)$, $m\>1$.

\subsection{}
\label{Par1}

Let $\pit$ be an irreducible $\ell$-adic representation of $\G$.
Fix a representative $(\M,\rt)$ in the inertial class of its cuspidal support, 
with $\M$ a standard Levi subgroup 
$\GL_{m_1}(\D)\tdt\GL_{m_r}(\D)$ and $\rt$ of the form $\rt_1\odo\rt_r$,
where $\rt_i$ is an $\ell$-adic cuspidal irreducible representation of 
$\GL_{m_i}(\D)$, 
for $i\in\{1,\dots,r\}$, and with $m_1+\dots+m_r=m$.
Since $\rt_i$ is determined up to an unramified twist, we may assume it is
integral, and fix an irreducible subquotient $\rho_i$ of its reduction 
modulo~$\ell$. 
By the classification of mod $\ell$ irreducible cuspidal representations in 
terms of supercuspidal re\-pre\-sentations \cite[Théorème~6.14]{MSc}, 
there are a unique integer $u_i\>1$ dividing $m_i$ and a supercuspidal irreducible 
representation $\tau_i$ of degree $u_i$
such that the supercuspidal support of $\rho_i$ is inertially equivalent to
\begin{equation*}
(\GL_{u_i}(\D)\tdt\GL_{u_i}(\D),\tau_i\odo\tau_i)
\end{equation*}
where the factors are repeated $k_i$ times, with $m_i=k_iu_i$.

\begin{defi}
\label{defmodliss}
Let $\pit$ be an irreducible $\ell$-adic representation of $\G$ as above.
Let us write
\begin{equation*}
\L = \GL_{u_1}(\D)^{k_1}\tdt\GL_{u_r}(\D)^{k_r}, 
\quad
\tau = \underbrace{\tau_1\odo\tau_1}_{k_1 \text{ 
    times}}\odo\underbrace{\tau_r\odo\tau_r}_{k_r \text{ times}}.
\end{equation*}
The inertial class in $\G$ of the supercuspidal pair $(\L,\tau)$, 
denoted $\il(\pit)$, is uniquely determined by the irreducible representation $\pit$. 
It is called the 
\emph{mod $\ell$ inertial supercuspidal support} 
of $\pit$.
\end{defi}

\begin{defi}
\label{lblock}
Two irreducible $\ell$-adic representations $\pit_1,\pit_2$ of $\G$ are said 
to \emph{belong to the same $\ell$-block} if $\il(\pit_1)=\il(\pit_2)$. 

An \emph{$\ell$-block} in the set $\XA(\G,\qlb)$ 
of all isomorphism classes of irreducible $\ell$-adic representations of $\G$
is an equivalence class for the equivalence relation defined by $\il$. 
\end{defi}

Let $\pit$ be an irreducible $\ell$-adic representation of $\G$ as above.
By definition, $\il(\pit)$ depends only on the inertial class of the 
supercuspidal support of $\pit$.
Assume now $\pit$ is integral. 

\begin{lemm}
All irreducible subquotients occurring in $\rl(\pit)$, 
the reduction mod $\ell$ of $\pit$,
have their supercuspidal support in $\il(\pit)$.
\end{lemm}

\begin{proof}
The representation $\pit$ is a subquotient of $\rt_1\tdt\rt_r$.
Since $\pit$ is integral, all the $\rt_i$'s are integral and, 
by Proposition \ref{redcusp}, for each $i$ there is an integer $a_i\>1$ such that
\begin{equation*}
\r_\ell(\rt_i) = \rho_i+\rho_i\nu+\dots+\rho_i\nu^{a_i-1},
\end{equation*}
where $\nu$ denotes the unramified mod $\ell$ character ``absolute value of the
reduced norm''. 
Thus any irreducible subquotient of $\rl(\pit)$ occurs as a subquotient of 
$\rho_1\nu^{i_1}\tdt\rho_r\nu^{i_r}$ for some 
integers $i_1,\dots,i_r\in\NN$. 
The result now follows by looking at the supercuspidal support of each $\rho_i$. 
\end{proof}

\begin{coro}
\label{C1}
Any two integral irreducible $\ell$-adic representation of $\G$ whose 
reductions mod $\ell$ share a common irreducible component belong to the 
same $\ell$-block. 
\end{coro}

\subsection{}

Let $\pit$ be a simple irreducible $\ell$-adic representation of $\G$. 
There are an integer $r\>1$ dividing~$m$ and a cuspi\-dal irreducible representation 
$\rt$ of $\G_{m/r}$ such that the inertial class of its cuspidal support 
contains 
\begin{equation}
\label{ICQLB}
(\GL_{m/r}(\D)^r,\rt\odo\rt).
\end{equation}
We may assume $\rt$ is integral.
We fix an irreducible subquotient $\rho$ of its reduction modulo $\ell$. 
As in Paragraph \ref{Par1}, there are a unique integer $u\>1$ dividing 
$m/r$ and a supercuspidal irreducible representation $\tau$ of 
degree $u$ such that the supercuspidal support of $\rho$ is inertially 
equivalent to 
$(\GL_{u}(\D)\tdt\GL_{u}(\D),\tau\odo\tau)$,
with $m=kur$.
Therefore, the mod $\ell$ inertial supercuspidal support $\il(\pit)$ 
of the $\ell$-adic discrete series representation $\pit$ is the inertial 
class of the pair 
\begin{equation}
\label{prototypeisc}
(\GL_{u}(\D)^{kr},\tau\odo\tau).
\end{equation}

Recall that, 
according to \cite[Théorème~6.11]{MSc}, any supercuspidal irreducible mod $\ell$ 
representation can be lifted to an $\ell$-adic irreducible representation. 
The following lemma is an immediate consequence of the definition of the 
mod $\ell$ inertial supercuspidal support. 

\begin{lemm}
\label{C2}
Let $\tt$ be an $\ell$-adic lift of $\tau$.
Any simple irreducible $\ell$-adic representation 
whose cuspidal support is inertially equivalent to 
\begin{equation*}
(\GL_{u}(\D)^{kr},\tt\odo\tt)
\end{equation*}
is in the same $\ell$-block as $\pit$.
In particular, the $\ell$-adic discrete series representation 
$\L(\tt,kr)$ is in the same $\ell$-block as $\pit$. 
\end{lemm}

\subsection{}
\label{Oriol}

{
Recall that we have fixed in Section \ref{Notation}
a smooth additive character 
${\psi}_\ell^{}:\F\to\overline{\mathbf{Q}}{}_\ell^\times$, 
trivial on $\mathfrak{p}$ but not on $\Oo$. 
We may and will assume that it has values in 
$\overline{\mathbf{Z}}{}_\ell^\times$.}
For any simple stratum $[\aa,\b]$~in $\Mat_m(\D)$,
the set of simple $\ell$-adic characters 
associated with $[\aa,\b]$ will be defined with respect to this choice 
(see Paragraph \ref{Par5}),
whereas the set of $\ell$-modular simple characters 
associated with $[\aa,\b]$ will be defined with respect to the reduction mod 
$\ell$ of ${\psi}_\ell^{}$.
Reduction mod~$\ell$ thus induces a bijection between $\ell$-adic 
and $\ell$-modular simple characters associated with $[\aa,\b]$. 
It also induces a bijection between endo-classes of $\ell$-adic 
and $\ell$-modular simple characters. 
Thus we will speak of endo-classes of simple characters, without 
referring to the coefficient field.

Write $\XLT$ for the group of $\qlb$-characters of $\mult\kk$,
and fix a maximal $\ell$-adic $\b$-extension $\kt$ of~$\G$.
The map \eqref{MAPBL} gives us a bijection $\blt_{\ell}$
from $\Ga\backslash\XLT$ onto the set of $\G$-conjugacy classes of 
simple $\ell$-adic super\-types 
of $\G$ with endo-class $\TT$. 
Also write $\XL$ for the group of $\flb$-characters of $\mult\kk$,
and $\k$ for the reduction mod $\ell$ of $\kt$.
This gives us a bijection 
$\bl_{\ell}$ from $\Ga\backslash\XL$ onto the set of $\G$-conjugacy classes of 
simple mod $\ell$ supertypes 
of $\G$ with endo-class $\TT$. 
These two bijections are compatible in the following sense. 

\begin{prop}
\label{Gontran}
Let $\pit$ be a simple irre\-du\-cible $\ell$-adic representation of $\G$ with 
endo-class~$\TT$, let $\car\in\XLT(\pit)$ and let 
$\phi\in\XL$ be the reduction mod $\ell$ of $\car$.
Then the inertial class $\il(\pit)$ corresponds through \eqref{BIJfk} and 
\eqref{PARISSC}
to the $\G$-conjugacy class of the simple supertype $\bl_{\ell}(\phi)$.
\end{prop}

\begin{proof}
Write the inertial class of the cuspidal sup\-port of $\pit$ as in \eqref{ICQLB}. 
Let $r$ be the degree of $\kk$ over $\dd[\car]$ and 
$\st_0$ be the $\ell$-adic supercuspidal representation of 
$\GL_{\u}(\kk)$ associated with $\car$, where $m'=ru$.
There is a maximal $\b$-extension $\kt_0$ of $\GL_{m/r}(\D)$ 
such that $\kt_0\otimes\st_0$ is a maximal simple type contained in 
$\rt$.
Let $\rho$ be an irreducible factor of the reduction mod $\ell$ of $\rt$. 
Then $\rho$ contains the maximal simple type $\k_0\otimes\s_0$,
where $\k_0$ is the reduction mod $\ell$ of $\k_0$ and $\s_0$ is that of 
$\st_0$.

Let $t$ be the degree of $\kk$ over $\dd[\phi]$.
By \cite[Lemme~3.2]{MSj}, 
if we write $\rho$ in the form $\Sp(\tau,k)$, with $\tau$ 
supercuspidal (see Proposition \ref{Coupure}), 
then $kr=t$ and $\s_0$ is the unique nondegenerate irreducible subquotient of 
the induced representation $\s_1\tdt\s_1$, 
where 
$\s_1$ is the supercuspidal mod $\ell$ re\-pre\-sentation of $\GL_{m'/t}(\dd)$ 
corresponding to $\phi$.
Moreover, there exists a maximal $\b$-extension $\k_1$ of $\GL_{m/t}(\D)$ 
such that $\k_1\otimes\s_1$ is a maximal simple type contained in 
$\tau$.
\end{proof}

Given $\car\in\X$, write $[\car]$ for its $\Ga$-orbit and $\phi$ for 
its reduction mod $\ell$.
The orbit $[\phi]$ depends only on $[\car]$, and is called the 
reduction mod $\ell$ of $[\car]$. 

\begin{prop}
\label{beaufort}
Two simple irreducible $\ell$-adic representations of $\G$ of endo-class $\TT$ 
are in the same $\ell$-block if and only if their parametrizing classes have 
the same reduction mod $\ell$.
\end{prop}

\begin{proof}
This follows from Propositions \ref{LaSURJ} and \ref{Gontran}. 
\end{proof}

We keep in mind the straightforward but important following fact.

\begin{rema}
\label{ofcourse}
Two simple irreducible $\ell$-adic representations of $\G$ in the same 
$\ell$-block have the same endo-class.
\end{rema}

\section{Linked complex representations}

In this section, 
we write $\G=\GL_m(\D)$, $m\>1$.

\subsection{}

We fix a prime number $\ell$ different from $p$
and an isomorphism of fields $\iota_\ell:\CC\simeq\qlb$.
If $\pi$ is a com\-plex representation of $\G$, 
write $\iota_\ell^*\pi$ for the $\ell$-adic representation of $\G$ obtained 
by extending scalars from $\CC$ to $\qlb$ along $\iota_\ell$.

\begin{defi}
\label{defllinked}
Two irreducible complex representations $\pi_1$, $\pi_2$ of $\G$ 
are said to be \emph{$\ell$-linked} if 
the irreducible $\ell$-adic representations 
$\iota_\ell^*\pi_1$ and $\iota_\ell^*\pi_2$ are in the 
same $\ell$-block.
\end{defi}

\begin{lemm}
\label{doesnotdepend}
This definition does not depend on the choice of $\iota_\ell$.
\end{lemm}

\begin{proof}
It is enough to prove that,
for any field automorphism $\t\in\Aut(\qlb)$, 
two simple $\ell$-adic representations 
$\pit_1$ and $\pit_2$ of $\G$ are in the same $\ell$-block if and only if 
$\pit_1^\t$ and $\pit_2^\t$ are in the same $\ell$-block.
By twisting by an unramified character of $\G$ if necessary,
we may assume~that~$\pit_1$ and $\pit_2$ are integral
and that their central characters have finite order. 
Thus the central charac\-ters of $\pit_1^\t$ and $\pit_2^\t$ also have finite 
order, which implies that they are integral too.

Given a simple $\ell$-adic representation $\pit$,
write an element of its mod $\ell$ inertial supercuspidal support 
as in \eqref{prototypeisc}.
Then the mod $\ell$ inertial supercuspidal support of $\pit^\t$ is the 
inertial class of the pair 
\begin{equation*}
(\GL_{u}(\D)^{kr},\tau^\t\odo\tau^\t).
\end{equation*}
The result follows.
\end{proof}

\subsection{}

Recall that we have fixed in Section \ref{Notation}
a smooth additive character 
${\psi}:\F\to\CC^\times$, 
trivial on~$\mathfrak{p}$ but not on $\Oo$. 
For any simple stratum $[\aa,\b]$,
the set of simple complex characters 
associated with $[\aa,\b]$ will be defined with respect to this choice 
(see Paragraphs \ref{Par5} and \ref{Oriol}).
Moreover, we assume the character $\iota_\ell\circ\psi$ is the character 
$\psi_\ell$ used in Paragraph \ref{Oriol}.
This choice gives us a bijection between endo-classes of complex 
and $\ell$-adic simple characters of $\G$. 
Again, we will speak of endo-classes of simple characters, without 
referring to the coefficient field.

Let $\k$ be a maximal complex $\b$-extension of $\G$ with endo-class $\TT$. 
Write $\X$ for the group of complex characters of $\mult\kk$.

\begin{lemm}
\label{carl}
Let $\pi$ be a simple irreducible complex representation of $\G$ with 
endo-class~$\TT$.
Then we have 
\begin{equation*}
\car\in\X(\k,\pi) 
\quad\Leftrightarrow\quad
\iota_\ell^{}\circ\car\in\X_\ell^{}(\iota_\ell^*\k,\iota_\ell^*\pi).
\end{equation*}
\end{lemm}

\begin{proof}
We have $\car\in\X(\k,\pi)$ if and only if $\pi$ contains the simple type 
$\bl(\car)=\bk(\car)\otimes\bs(\car)$, which occurs if and only if 
$\iota_\ell^*\pi$ contains the $\ell$-adic simple type 
$\iota_\ell^*\bl(\car)$.
Thus it suffices to prove that $\iota_\ell^*\bl(\car)$ is equal to 
$\blt_\ell(\iota_\ell\circ\car)$, 
{where $\blt_\ell$ is the map as in 
Paragraph \ref{Oriol} defined 
with respect to the maximal $\b$-extension $\iota_\ell^*\k$. }

Firstly,
the $\ell$-adic $\b$-extension $\bkt_\ell(\iota_\ell\circ\car)$ 
associated with $\iota_\ell\circ\car$ with respect to $\iota_\ell^*\k$ 
is equal to $\iota_\ell^*\bk(\car)$.
Secondly, the $\ell$-adic supercuspidal representation 
$\bst_\ell(\iota_\ell\circ\car)$ associated with $\iota_\ell\circ\car$ 
(with respect to the choice of an isomorphism \eqref{ISOB}) 
is equal to $\iota_\ell^*\bs(\car)$, since it is characterized by a trace 
formula (see Remark \ref{PetrusMartel}).
The result follows.
\end{proof}

\begin{defi}
Let $\car\in\X$.
The \emph{$\ell$-regular part} of $\car$ is the 
unique complex character $\car_\ell\in\X$ whose order is prime to $\ell$ and
such that $\car\car_\ell^{-1}$ has order a power of $\ell$.
\end{defi}

Given $\a\in\X$,
the orbit $[\car_\ell]$ depends only on $[\car]$.
It is called the $\ell$-regular part of $[\car]$, denoted $[\car]_\ell$. 

\begin{prop}
\label{AdeleGodel}
Two simple irreducible complex representations of $\G$ with endo-class $\TT$ 
are $\ell$-linked if and only if the $\ell$-regular 
parts of their parametrizing classes are equal. 
\end{prop}

\begin{proof}
Let $\pi_1,\pi_2$ be simple irreducible complex representations of $\G$ with 
endo-class $\TT$. 
We fix $\car_i\in\X(\k,\pi_i)$ for each $i=1,2$.
By Lemma \ref{carl} and Proposition \ref{beaufort}, 
the representations $\pi_1,\pi_2$ are $\ell$-linked if
and only if $[\iota_\ell\circ\car_1]$ and $[\iota_\ell\circ\car_2]$ 
have the same reduction mod $\ell$.
But the reduction mod $\ell$ of $[\iota_\ell\circ\car]$, for a character 
$\a\in\X$, is the same as that of $[\iota_\ell\circ\car_\ell]$.
It follows that we have 
$[\iota_\ell\circ(\car_1)_\ell]=[\iota_\ell\circ(\car_2)_\ell]$, 
thus $[\car_1]_\ell=[\car_2]_\ell$.
\end{proof}

\subsection{}

Write $q$ for the cardinality of the residue field of $\F$.
For each prime number $\ell$ dividing
\begin{equation}
\label{prodqi1}
(q^n-1)(q^{n-1}-1)\dots(q-1)
\end{equation}
we fix an isomorphism of fields $\iota_\ell:\CC\simeq\qlb$.

\begin{defi}
\label{deflinked}
Two irreducible complex representations $\pi,\pi'$ of $\G$ 
are \emph{linked} 
if there are~a finite family $\ell_1,\dots,\ell_r$ of prime numbers dividing 
\eqref{prodqi1}
and a finite family of irreducible complex representations 
$\pi=\pi_0,\pi_1,\dots,\pi_r=\pi'$ such that,
for all integers $i\in\{1,\dots,r\}$, the representations $\pi_{i-1}$ and $\pi_i$ are 
$\ell_i$-linked.
\end{defi}

\begin{rema}
By Lemma \ref{doesnotdepend},
this does not depend on the choice of the isomorphisms 
$\iota_\ell$ for $\ell$ dividing \eqref{prodqi1}.
\end{rema}

Two linked simple complex representations of $\G$ 
have the same endo-class (see Remark \ref{ofcourse}). 
The converse is given by the following proposition.

\begin{prop}
\label{MAINTHEOREM1}
Two simple irreducible complex representations are linked if and 
only if they have the same endo-class. 
\end{prop}

\begin{proof}
Assume $\pi$ and $\pi'$ are simple irreducible complex representations 
with the same endo-class $\TT$.
Let $\car$ and $\car'$ be characters in $\X(\pi)$ and $\X(\pi')$, 
respectively, and write $\xi={\car'}{\car^{-1}}$.
Let $\ell_1,\dots,\ell_r$ be the prime numbers dividing \eqref{prodqi1}.
The character $\xi$ decomposes uniquely as
\begin{equation*}
\xi = \xi_1\dots\xi_r
\end{equation*}
where the order of $\xi_i$ is a power of $\ell_i$, 
for $i\in\{1,\dots,r\}$.
Write $\car_0=\car$ and define inductively
\begin{equation*}
\car_i=\car_{i-1}\cdot\xi_i
\end{equation*}
for all $i\in\{1,\dots,r\}$.
Let $\pi_i$ be a simple irreducible complex representation of endo-class $\TT$ 
and parametrizing class $[\car_i]$.
The result follows from Proposition \ref{AdeleGodel}. 
\end{proof}

\begin{rema}
\label{Solgrub}
Suppose that $\pi$ and $\pi'$ are discrete series representations with the 
same endo-class. 
The proof of Proposition \ref{MAINTHEOREM1} shows that the simple 
representations $\pi_1,\dots,\pi_{r-1}$ linking $\pi$ to $\pi'$ can be chosen 
to be discrete series representations as well.
\end{rema}

\subsection{}

Let $\pi$ be an irreducible complex representation of $\G$.
Fix a representative $(\M,\rho)$ in its cuspidal support, 
with $\M=\GL_{m_1}(\D)\tdt\GL_{m_r}(\D)$ 
and $\rho=\rho_1\odo\rho_r$,
with $m_1+\dots+m_r=m$, and
where $\rho_i$ is a cuspidal irre\-ducible representation of $\GL_{m_i}(\D)$
for $i\in\{1,\dots,r\}$.
Write $\TT_i$ for the endo-class of $\rho_i$ and 
$g_i$ for the degree of $\TT_i$.
We define the \emph{semi-simple endo-class} of $\pi$ to~be the formal sum
\begin{equation}
\label{SSEC}
\TT(\pi) = \sum\limits_{i=1}^{r} \frac{m_id}{g_i}\cdot\TT_i
\end{equation}
in the free abelian semigroup generated by all $\F$-endo-classes. 
It depends only on the inertial class of the cuspidal support of $\pi$. 

Note that, if $\pi$ is a simple irreducible representation with endo-class 
$\TT$, then its semi-simple endo-class is $\TT(\pi)=ng^{-1}\cdot\TT$ where 
$g$ is the  degree of $\TT$.

The following theorem, which is our first main result,
generalizes Proposition \ref{MAINTHEOREM1}.

\begin{theo}
\label{MAINTHEOREM11}
Two irreducible complex representations are linked if and 
only if they have the same semi-simple endo-class. 
\end{theo}

\begin{proof}
Any two linked irreducible complex representations automatically 
have the same semi-simple endo-class. 
We thus start with two irreducible complex representations $\pi$, $\pi'$ 
with the same semi-simple endo-class.
By \cite[Théorème~4.16]{MSt}, the representation $\pi$ can be written 
\begin{equation*}
\pi = \pi_1\times\pi_2\times\dots\times\pi_k
\end{equation*}
where $\pi_1,\pi_2,\dots,\pi_k$ 
are simple irreducible representations 
whose inertial cuspidal supports are pairwise distinct, 
 and this decomposition is unique up to renumbering. 
We have the following straightforward lemma.

\begin{lemm}
\label{glissement}
Let $\delta$ be an irreducible complex representation of $\GL_{m-k}(\D)$ 
for some integer $k\in\{1,\dots,m-1\}$.
Let $\s$, $\s'$ be two irreducible complex representations of 
$\GL_{k}(\D)$,
and let $\pi$,~$\pi'$ be irre\-ducible subquotients of 
$\s \times \d$ and $\s' \times \d$, respectively.
If $\s$ and $\s'$ are linked, then $\pi$ and~$\pi'$ are linked.
\end{lemm}

For each $i\in\{1,\dots,k\}$,
thanks to Lemma \ref{glissement} and Proposition \ref{MAINTHEOREM1}, 
we may and will assume that 
$\pi_i$ is a discrete series representation of the form 
$\L(\rho_i,r_i)$ for some cuspidal representation~$\rho_i$ of 
$\GL_{m_i}(\D)$ with same endo-class as $\pi_i$
and some integer $r_i$, such that 
$m_1r_1+\dots+m_kr_k=m$.
We may even assume that $\rho_i$ has minimal degree among all 
cuspidal irreducible representations of $\GL_{a}(\D)$, $a\>1$,
with the same endo-class as $\pi_i$.
This amounts to saying that $m_i$ is equal to $g_i/(g_i,d)$, where $g_i$ is the 
degree of the endo-class of $\pi_i$.

Moreover, if $\rho_i$, $\rho_j$ have the same endo-class for some 
$i,j\in\{1,\dots,k\}$, 
then they have the same degree, thus they are linked.
We thus may assume 
$\rho_1,\dots,\rho_k$ have distinct
endo-classes, denoted $\TT_1,\dots,\TT_k$, respectively.

Similarly, we may assume the representation $\pi'$ decomposes as a product 
$\pi'_1\times\pi'_2\times\dots\times\pi'_t$,
where $\pi'_j$ is a discrete series representation of the form 
$\L(\rho_j',s_j^{})$ for some cuspidal representation $\rho'_j$ of 
$\GL_{m_j'}(\D)$ and some integer $s_j^{}\>1$,
and we may assume that the endo-classes $\TT'_1,\dots,\TT'_t$ 
of $\rho_1',\dots,\rho_t'$ are distinct.
It follows that $k=t$ and, up to renumbering, we may assume that we have
$\TT_i'=\TT_i^{}$ for~each $i\in\{1,\dots,k\}$.
It then follows that $\rho'_i$ and $\rho_i^{}$ have the same degree, 
by the minimality of $m_i$.

Since $\pi$ and $\pi'$ have the same semi-simple endo-class,
we have $s_i=r_i$ for all $i$, thus $\pi_i^{}$ and~$\pi'_i$ 
have the same degree.
Proposition \ref{MAINTHEOREM1} then implies that 
$\pi_i^{}$ and~$\pi'_i$ are linked.
Theorem \ref{MAINTHEOREM11}~now follows from Lemma \ref{glissement} again.
\end{proof}

\section{Application to the local Jacquet--Langlands correspondence} 

We fix $n=md$ and write $\G=\GL_m(\D)$ and $\H=\GL_{n}(\F)$. 
As in the introduction, we write $\Dd(\G,\CC)$ for the set of all isomorphism 
classes of complex discrete series representations of $\G$, and similarly for 
$\H$. 
We write
\begin{equation}
\label{LJLC}
\boldsymbol{\pi} : 
\Dd(\G,\CC)\to\Dd(\H,\CC)
\end{equation}
for the local Jacquet--Langlands correspondence. 

\subsection{}

We fix an isomorphism of fields $\iota_\ell:\CC\simeq\qlb$ 
and write (as in \cite{MSj})
\begin{equation}
\label{corrjlt}
\plt : \Dd(\G,\qlb) \to \Dd(\H,\qlb)
\end{equation}
for the $\ell$-adic local Jacquet--Langlands correspondence between $\ell$-adic 
discrete series representa\-tions of $\G$ and $\H$.
The correspondence \eqref{corrjlt} 
does not depends on the choice of $\iota_\ell$
(\cite[Remar\-que~10.1]{MSj}).
According to \cite[Paragraph~3.1]{BaduJIMJ}, 
there is a unique surjective group homomorphism
\begin{equation*}
\BLT : \RA(\H,\qlb) \to \RA(\G,\qlb)
\end{equation*}
where $\RA(\G,\qlb)$ is the Grothendieck group of finite length 
$\ell$-adic representations of $\G$, with the following property: 
given positive integers $n_1,\dots,n_r$ such that $n_1+\dots+n_r=n$ 
and an $\ell$-adic discrete series representation $\st_i$ of 
$\GL_{n_i}(\F)$ for each $i$, 
the image of the product $\st_1\tdt\st_r$ by $\BLT$ is $0$ if $n_i$ is not 
divisible by $d$ for at least one $i$, and is $\pit_1\tdt\pit_r$ otherwise, 
where $n_i=m_id$ and $\pit_i$ is the $\ell$-adic discrete series representation 
of $\GL_{m_i}(\D)$ whose Jacquet--Langlands transfer is $\st_i$, for each $i$. 

By \cite[Théorème~12.4]{MSj}, there exists a unique surjective 
group homomorphism of Grothendieck groups
$\BL : \RA(\H,\flb) \to \RA(\G,\flb)$ 
such that the diagram 
\begin{equation*}
\begin{CD}
{\RA}(\H,\qlb)^{{\rm e}} @>{\BLT}>>
{\RA}(\G,\qlb)^{{\rm e}} \\
@V{\r_\ell}VV @VV {\r_\ell} V \\
\RA(\H,\flb)^{\phantom{\rm e}} @>>{\BL}> \RA(\G,\flb)^{\phantom{\rm e}} \\
\end{CD}
\end{equation*}
is commutative, where ${\RA}(\G,\qlb)^{{\rm e}}$ is the subgroup of 
$\RA(\G,\qlb)$ generated by integral irreducible representations,
and $\RA(\G,\flb)$ is the Grothendieck group of 
$\ell$-modular representations of $\G$.

\begin{prop}
\label{lLINKEDJL}
Let $\pit_1$ and $\pit_2$ be $\ell$-adic discrete series representations of 
$\G$, and write $\st_1$, $\st_2$ for their Jacquet--Langlands transfers to 
$\H$, respectively. 
If $\st_1$, $\st_2$ are in the same $\ell$-block of $\H$, 
then $\pit_1$, $\pit_2$ are in the same $\ell$-block of $\G$.
\end{prop}

\begin{proof}
Let us write $\st_i=\L(\rt_i,r_i)$ and $k_i=k(\rt_i)$ for $i=1,2$.
Then $k_1r_1=k_2r_2$, which we denote by $v$, 
and the mod $\ell$ inertial supercuspidal support of $\st_1$ 
and $\st_2$ contains the supercuspidal pair
\begin{equation*}
(\GL_{u}(\F)\tdt\GL_u(\F),\tau\odo\tau),
\end{equation*}
with $uv=m$ and for some mod $\ell$ supercuspidal representation $\tau$ 
of $\GL_{u}(\D)$.
Fix an $\ell$-adic lift $\tt$ of $\tau$ and write $\st=\L(\tt,v)$.
The representation $\st$ is in the same $\ell$-block as 
$\st_1,\st_2$, by Lemma \ref{C2}.
If we write $\pit$ for the $\ell$-adic discrete series representation of $\G$ 
whose transfer to $\H$ is $\st$, 
then it is enough to prove that $\pit$ is in the same $\ell$-block 
as $\pit_1$. 

In the remainder of the proof, it will be more convenient for us to deal with Speh 
representa\-tions rather than discrete series representations, as in \cite{MSj}. 
We thus apply the Zelevinski involution to $\pit$, $\pit_1$ and $\st$, $\st_1$ 
and thus get $\ell$-adic Speh representa\-tions.

Let us write $\st^*$ for the Zelevinski dual of $\st$.
Its reduction mod $\ell$ is the $\ell$-modular super-Speh representation 
$\Z(\tau,v)$, by \cite[Théorème~9.39]{MSc}. 
If we write $\pit^*=\Z(\widetilde{\a},t)$ for the Zelevinski dual of $\pit$, 
for some $t$ dividing $m$ and some cuspidal irreducible representation 
$\widetilde{\a}$ of $\GL_{m/t}(\D)$, then its reduction mod $\ell$ contains the 
Speh representation $\Z(\a,t)$ where $\a$ is an irreducible component~of the 
reduction mod $\ell$ of $\widetilde{\a}$
(see for instance \cite[Proposition~1.10]{MSj}).
The cuspidal representation~$\a$ need not be supercuspidal but, 
according to Proposition \ref{Coupure}, 
it can be written as $\Sp(\b,k)$ for $k=k(\a)$ and some 
supercuspidal irreducible representation $\b$.

We now look at the reduction mod $\ell$ of the Zelevinski dual of 
$\st_1$. 
It is $\Z(\rho_1,r_1)$ where $\rho_1$, the reduction mod $\ell$ of $\rt_1$,
can be written as $\Sp(\tau\chi,k_1)$ for some unramified character $\chi$.
By twisting $\pit_1$ by an unramified character of $\G$, 
we may assume that $\chi$ 
is trivial. 
According to \cite[Lemme~9.41]{MSc}, the represen\-ta\-tion $\Z(\rho_1,k_1)$ 
decomposes as a $\ZZ$-linear combination of products of the form 
\begin{equation*}
\Z(\tau\nu^{i_1},v_1)\tdt\Z(\tau\nu^{i_r},v_r)
\end{equation*}
with $v_1+\dots+v_r=v$ and $i_1,\dots,i_r\in\ZZ$, 
where $\nu$ stands for the absolute value of the reduced~norm,
as usual.
(For an explicit formula for this decomposition, 
see \cite[Sections~11 and~12]{MSj}.)
Thanks to the commutative diagram above, 
the reduction modulo $\ell$ of the Zelevinski dual of $\pit_1$ will be made of 
products of the form 
\begin{equation*}
\Z(\a\nu^{i_1},t_1)\tdt\Z(\a\nu^{i_r},t_r)
\end{equation*}
with $t_1+\dots+t_r=t$ and $i_1,\dots,i_r\in\ZZ$, 
all of whose irreducible subquotients have supercuspidal support inertially
equivalent to
$(\GL_w(\D)\tdt\GL_w(\D),\b\odo\b)$,
with $wkt=m$.
The result follows from Corollary \ref{C1}.
\end{proof}

\subsection{}

Proposition \ref{lLINKEDJL} implies that two complex discrete series representations $\pi_1,\pi_2$ of $\G$ are linked if their Jacquet--Langlands transfers are linked. Then Proposition \ref{MAINTHEOREM1} (together with Remark \ref{Solgrub}) induces a map
\begin{equation*}
\boldsymbol{\pi}_1 : \Ee_n(\F) \to \Ee_n(\F)
\end{equation*}
depending on $\G$,
where $\Ee_n(\F)$ is the set of $\F$-endo-classes of degree dividing $n$. 
More precisely, given~an endo-class $\TT\in\Ee_n(\F)$ and 
a complex discrete series representation $\s$ of $\H$ of endo-class $\TT$, 
the endo-class of the Jacquet--Langlands transfer 
of $\s$ to $\G$ depends only on $\TT$: 
we denote it $\boldsymbol{\pi}_1(\TT)$.

This map does not depend on  the choice of the isomorphisms 
$\iota_\ell$ for $\ell$ dividing \eqref{prodqi1}.

\begin{prop}
The map $\boldsymbol{\pi}_1$ is bijective. 
\end{prop}

\begin{proof}
This map preserves depth.
Thus, for any rational number $r\in\QQ_+$, we have a map
\begin{equation*}
\boldsymbol{\pi}_{1,r} : \Ee_n(\F,r) \to \Ee_n(\F,r)
\end{equation*}
where $\Ee_n(\F,r)$ the set of of $\F$-endo-classes of degree dividing $n$ and 
depth $r$, which we claim to be surjective.
Indeed, given $\TT\in\Ee_n(\F,r)$, let us choose 
$\pi\in\XA(\G,\CC)$ of endo-class $\TT$,
and let $\s$ be its Jacquet--Langlands transfer to $\H$. 
Then the endo-class of $\s$ is an antecedent of $\TT$ by 
$\boldsymbol{\pi}_1$. 
Since $\Ee_n(\F,r)$ is a finite set and $\boldsymbol{\pi}_{1,r}$ is 
surjective, it follows that 
$\boldsymbol{\pi}_{1,r}$ is bijective for all $r$, 
thus the map $\boldsymbol{\pi}_1$ is bijective.
\end{proof}

As a corollary, we have the following refinement of Proposition \ref{lLINKEDJL}. 

\begin{coro}
\label{lLINKEDJLcoro}
Let $\pit_1$ and $\pit_2$ be $\ell$-adic discrete series representations of 
$\G$, and write $\st_1$, $\st_2$ for their Jacquet--Langlands transfers to 
$\H$, respectively. 
Then $\st_1$, $\st_2$ are in the same $\ell$-block of $\H$
if and only if $\pit_1$, $\pit_2$ are in the same $\ell$-block of $\G$.
\end{coro}

Allowing $\ell$ to vary, we deduce 

\begin{theo}
\label{TheoremC}
Two complex discrete series representations of\/ $\G$ are linked if and only 
if their transfers to\/ $\H$ are linked.
\end{theo}

Recall that the parametric degree of 
a cuspidal representation of $\G$
has been defined in \ref{P21}.

\begin{prop}\label{propintro}
For every complex discrete series representation of\/ $\G$, there is a 
cuspidal 
complex representation of\/ $\G$ with the same endo-class and with parametric 
degree\/ $n$. 
\end{prop}

\begin{proof}
Let $\pi$ be a complex discrete series representation of $\G$ with endo-class 
$\TT$. To find~a complex cuspidal representation 
with same endo-class and parametric degree $n$, we need to find a 
$\Gal(\kk/\dd)$-regular complex character 
$\car\in\X$ which is also $\Gal(\kk/\ee)$-reg\-ular.
The latter implies the former, so let us find a $\Gal(\kk/\ee)$-regular 
character $\car\in\X$.
For this, it is enough 
to choose for $\car$ a generator of the cyclic group $\X$. 
\end{proof}

As an immediate consequence, we get the following.

\begin{theo}
\label{unseulTT}
Given an endo-class $\TT$ in $\Ee_n(\F)$,
if there is a complex cuspidal representation $\rho$ of $\G$ 
with endo-class $\TT$ and parametric degree $n$ such that 
$\boldsymbol{\pi}(\rho)$ has endo-class $\TT$, then 
$\boldsymbol{\pi}_1(\TT)$ is equal to $\TT$.
\end{theo}

\begin{coro}
\label{coroesstame}
Assume that $\TT\in\Ee_n(\F)$ is essentially tame, that is, its ramification 
order~is prime to $p$.
Then, for all complex discrete series representation $\pi$ of $\G$ with endo-class 
$\TT$,~the~re\-presentations $\pi$ and ${}_{{\rm JL}}\pi$ have the same endo-class. 
\end{coro}

\begin{proof}
This follows from Theorem \ref{unseulTT} together with the main result of 
\cite{BHJL3}. 
\end{proof}

\begin{rema}
Let $\TT\in\Ee_n(\F)$ and write $\TT'=\boldsymbol{\pi}_1(\TT)\in\Ee_n(\F)$. 
Let $\rho$ be a cuspidal irreducible representation of $\G$ with endo-class 
$\TT$ and parametric degree $n$.
Its Jacquet--Langlands transfer to $\GL_n(\F)$ is a cuspidal representation 
denoted $\s$.
Then, for any $a\>1$, the discrete series~repre\-sentation $\L(\rho,a)$ of 
$\GL_{am}(\D)$ has endo-class $\TT$, and its 
transfer $\L(\s,a)$ to $\GL_{an}(\F)$ has endo-class $\TT'$. 
Thus $\boldsymbol{\pi}_1(\TT)$ does not depend on the choice of the integer 
$n\>1$ such that $\TT\in\Ee_n(\F)$.
\end{rema}

\section{Reduction to the maximal totally ramified case}

We continue with the previous notation, so that $\G=\GL_m(\D)$ and 
$\H=\GL_{n}(\F)$. 
In this~sec\-tion, we closely follow the ideas of~\cite[Section~6]{BHJL3} to 
make a further reduction to the maximal~to\-tal\-ly ramified case 
{(see Paragraph 1.1).} 
All representations in this section are com\-plex. 

\subsection{}

Let $\pi$ be a cuspidal (complex) representation of $\G$ with parametric degree $n$.
Let $(\BJ,\bl)$ be an extended maximal simple type of $\G$ contained in $\pi$
\cite[\S3.1 and Théorème~3.11]{MSt},
attached to a simple stratum $[\aa,\b]$ and a simple character $\t$.
We write $\B$ for the centralizer of $\b$ in $\Mat_m(\D)$, so that 
$\B \simeq \Mat_{m'}(\D')$, for some integer $m'\>1$ and $\F[\b]$-division 
algebra $\D'$. 
We fix a maximal unramified extension $\L$ of $\F[\b]$ in $\B$,
and write $\K$ for the maximal unramified subextension of $\L$ over $\F$.

We fix a root of unity $\zeta\in\K$ of order relatively prime to $p$
such that $\K=\F[\zeta]$. 
Write $\G_\K$ for the centralizer of $\K$ in $\G$.
Let $u$ be a pro-unipotent, elliptic regular element of $\G_\K$
in the sense of~\cite[Paragraph~1.6]{BHJL3}.
The element $h=\zeta u$ then lies in the set $\Ger$ of elliptic regular
elements of $\G$, so we have 
\begin{equation*}
\tr\ \pi(h) = \sum\limits_{x\in\G/\BJ}\tr\ \bl(x^{-1}hx)
\end{equation*}
as in~\cite[(6.3.1)]{BHJL3}. 
Write $\J=\J(\aa,\b)=\BJ\cap\aa^\times$.
A coset $x\BJ$ can only contribute to the sum if we have $x^{-1}hx\in\BJ$ or, 
equivalently, $x^{-1}hx\in\J$. 

Write $\Psi$ for the Galois group of $\K/\F$ and $\Ga$ for that of 
$\L/\F[\b]$. 
Restriction of operators~identi\-fies $\Ga$ with a subgroup of $\Psi$. 
Write $\Psit$ for the unique subgroup of $\Ga$ (thus of $\Psi$) 
of order $m'$.
As in~\cite[(6.3.2)]{BHJL3}, by using~\cite[6.3~Proposition]{BHJL3} we have
\begin{equation*}
\tr\ \pi(\zeta u) = \sum\limits_{\a\in\Psi/\Psit} \sum\limits_{y\in\G_\K/\BJ_\K}
\tr\ \bl(y^{-1}\zeta^\a u^\a y)
\end{equation*}
where $\BJ_\K=\BJ\cap\G_\K$.

Let us fix a uniformizer $\varpi_\F$ of $\F$.
We choose an irreducible representation $\bk$ of $\BJ$ such that:
\begin{enumerate}
\item
the restriction of $\bk$ to $\J$ is a $\b$-extension of $\t$;
\item
the character $\det(\bk)$ has order a power of $p$;
\item
the automorphism $\bk(\varpi_\F)$ is the identity.
\end{enumerate}
Note that such a representation is not unique.
We now write
\begin{equation*}
\bltt = \Hom_{\J^1}(\bk,\bl)
\end{equation*}
which carries an action of $\BJ$ given by 
$g\cdot f=\bl(g)\circ f\circ\bk(g)^{-1}$ for $g\in\BJ$ and $f\in\bltt$. 
This repre\-sentation is irreducible and trivial on $\J^1=\J^1(\aa,\b)$, 
and we have the decomposition $\bl=\bk\otimes\bltt$. 
As in~\cite[(6.4.1)]{BHJL3} this gives us
\begin{equation*}
\tr\ \pi(\zeta u) = \sum\limits_{\a\in\Psi/\Psit} \tr\ \bltt(\zeta^\a)
\sum\limits_{y\in\G_\K/\BJ_\K} \tr\ \bk(y^{-1}\zeta^\a u^\a y).
\end{equation*}
We are now going to interpret the sum over $\G_\K/\BJ_\K$ 
as the trace of a cuspidal irreducible representation of $\G_\K$.

\subsection{}

Write $\t_\K$ for the restriction of $\t$ to $\H^1(\aa,\b)\cap\G_\K$,
which is the interior $\K/\F$-lift of the simple character $\t$ in the sense of 
\cite[Section~5]{BSS}. 
The group $\BJ_\K$ is also the normalizer of $\t_\K$ in $\G_\K$.
We choose an irreducible representation $\bk_\K$ of $\BJ_\K$ such that:
\begin{enumerate}
\item
the restriction of $\bk_\K$ to $\J_\K$ is a $\b$-extension of $\t_\K$;
\item
the character $\det(\bk_\K)$ has order a power of $p$;
\item
the automorphism $\bk_\K(\varpi_\F)$ is the identity.
\end{enumerate}
Again, such a choice may not be unique.
The pair $(\BJ_\K,\bk_\K)$ is an extended maximal simple type in $\G_\K$.
It thus defines a cuspidal irreducible representation $\rho$ of 
$\G_\K$. 
By~\cite[(3.4.3) and~(5.6.2)]{BHETLC3},
there is a sign $\epsilon\in\{-1,+1\}$ such that
\begin{equation*}
\tr\ \bk(y^{-1}\zeta^\a u^\a y) = \epsilon\cdot\tr\ \eta_\K(y^{-1}\zeta^\a u^\a y)
\end{equation*}
where $\eta_\K$, which denotes the restriction of $\bk_\K$ to 
the pro-$p$-subgroup $\J^1\cap\G_\K$, 
is the Heisenberg representation of $\t_\K$.
As in~\cite[(6.4.2)]{BHJL3} this gives us
\begin{equation}
\label{ZUT}
\tr\ \pi(\zeta u) = \epsilon\sum\limits_{\a\in\Psi/\Psit} \tr\ \bltt(\zeta^\a) 
\ \tr\ \rho^{\a^{-1}}(u).
\end{equation}
We do not know whether a result similar to~\cite[6.5~Lemma]{BHJL3} holds,
that is, we do not know~whe\-ther the $\Psi$-stabilizers of $\rho$ and of its 
inertial class are both equal to $\Ga$. 
However, let $\Psi_0$ denote~the stabilizer in $\Psi$ of the inertial class of $\rho$ 
and let $\X_0$ be a set of representatives for $\Psi$ mod $\Psi_0$. 
For $\g\in\Psi_0$ there is an unramified character $\chi_\g$ 
of $\G_\K$ such that 
$\rho^{\g^{-1}}\simeq\rho\chi_\g$.
Since $u$ is pro-unipotent (thus compact) we have 
$\chi_\g^{\a^{-1}}(u)=1$, for all $\a\in\Psi/\Psit$.
Therefore~\eqref{ZUT} can be rewritten as 
\begin{equation}
\label{FORMULA0}
\tr\ \pi(\zeta u) = \epsilon \sum\limits_{\a\in\X_0} 
\tr\ \rho^{\a^{-1}}(u) 
\sum\limits_{\g\in\Psi_0/\Psit} 
\tr\ \bltt(\zeta^{\a\g}) 
\end{equation}
Note that the map
\begin{equation}
\label{DEFWAl}
{w} : \zeta \mapsto \sum\limits_{\g\in\Psi_0/\Psit} \tr\ \bltt(\zeta^{\g}) 
\end{equation}
is not identically zero on the set of $\K/\F$-regular roots of unity, 
by~\cite[Theorem~1.1(ii)]{SZ}.
{Thus there is an $\a\in\X_0$ such that the coefficient 
$w(\zeta^\a)$ in \eqref{FORMULA0} is nonzero.}

\subsection{}

Now write $\pi'$ for the Jacquet--Langlands transfer of $\pi$ to $\H$.
Since $\pi$ has parametric degree $n$, 
the torsion number $t(\pi)$ 
is equal to the degree of $\K$ over $\F$.
We now do for $\pi'$ what we did for $\pi$.

Let $(\BJ',\bl')$ be an extended maximal simple type of $\H$ contained in $\pi'$, 
attached to a simple~stra\-tum $[\aa',\b']$. 
Write $\B'$ for the centralizer of $\b'$ in $\Mat_n(\F)$,
fix a maximal unramified extension~$\L'$ of $\F[\b']$ in $\B'$
and write $\K'$ for the maximal unramified sub\-ex\-tension of $\L'$ over $\F$.
The relation $t(\pi)=t(\pi')$, together with the fact that $\pi'$ also has 
parametric degree $n$, implies that $\K'$ and $\K$ have the same degree over 
$\F$. 
Therefore, we may identify the maximal unramified sub\-ex\-tension of $\L'/\F$ 
with $\K$.

We have an analogue $\bltt'$ of $\bltt$ and an analogue $\rho'$ of $\rho$ in 
the argument of the previous paragraph so that we get 
\begin{equation*}
\tr\ \pi'(\zeta u') = \epsilon' \sum\limits_{\a'\in\X'_0} 
\tr\ \rho'^{{\a'}^{-1}}(u') 
\sum\limits_{\g'\in\Psi'_0/\Psit'} \tr\ \bltt'(\zeta^{\a'\g'}) 
\end{equation*}
where $\zeta\in\K$ is as above,
$u'$ is a pro-unipotent elliptic regular element
of the centralizer $\H_\K$~of~$\K$ in $\H$, 
$\epsilon'\in\{-1,+1\}$ is a sign and the subgroups
$\Psit',\Psi'_0$ and $\X_0'$ are defined as in the~previous~para\-graph. 
If $\zeta u'$ is chosen to have the same reduced 
characteristic polynomial over $\F$ as~$\zeta u$,~this 
is equal to $(-1)^{n-m}\cdot\tr\ \pi(\zeta u)$,
by the trace relation characterizing the Jacquet--Lang\-lands corres\-pondence.
We thus get: 
\begin{equation*}
\epsilon' \sum\limits_{\a'\in\X'_0} w'(\zeta^{\a'})\ \tr\ \rho'^{{\a'}^{-1}}(u')
=(-1)^{n-m}\cdot\epsilon 
\sum\limits_{\a\in\X_0} {w}(\zeta^{\a})\ \tr\ \rho^{\a^{-1}}(u) 
\end{equation*}
where the function $w$ and its analogue $w'$ are defined by \eqref{DEFWAl}.

We apply~\cite[6.6~Lemma]{BHJL3} (note that $\rho$ has maximal parametric 
degree since $\L/\K$ is maximal). 
The $\rho'^{\a'^{-1}}$, $\a'\in\X'_0$, are not unramified twists of each other, 
and the same holds for the~Jacquet--Langlands 
transfers to $\H_\K$ of the $\rho^{\a^{-1}}$, $\a\in\X_0$.
Thanks to linear independence of characters, it fol\-lows that there is a 
$\a\in\Psi$ such that
\begin{equation*}
\boldsymbol{\pi}_\K(\rho^{\a^{-1}}) = \rho'\chi
\end{equation*}
for some unramified character $\chi$ of $\H_\K$,
where $\boldsymbol{\pi}_\K$ is the local Jacquet--Langlands correspondence 
from $\G_\K$ to $\H_\K$. 

Assume now that $\boldsymbol{\pi}_\K$ preserves $\K$-endo-classes in the 
maximal totally ramified case.
Then~the representations $\rho^{\a^{-1}}$ and $\rho'$ have the same 
$\K$-endo-class.
But the $\K$-endo-class of $\rho^{\a^{-1}}$(respective\-ly, of~$\rho'$)
is a $\K/\F$-lift of the $\F$-endo-class of $\pi$ 
(respectively, of $\pi'$) in the sense of~\cite{BHLTL1}. 
Applying the restriction map from $\Ee(\K)$ onto $\Ee(\F)$,
we get that $\pi$ and $\pi'$ have the same $\F$-endo-class.~Thus 
we have proved Theorem~A of the introduction:

\begin{theo}
\label{MainTheoremA}
Assume that, for all $\F$ and $n$, and all maximal totally ramified, 
cusp\-idal~irre\-du\-cible complex representations $\rho$ of $\G$, 
the representations $\rho$ and
$\boldsymbol{\pi}(\rho)$ have the same endo-class.
Then the map $\boldsymbol{\pi}_1$ is the identity.
\end{theo}

\section{Explicit Jacquet--Langlands correspondence up to unramified twist} 
\label{Sec8}

Now let us fix~an endo-class $\TT\in\Ee_n(\F)$ and suppose that 
$\boldsymbol{\pi}_1(\TT)=\TT$. 
Write $\Dd_0(\G,\TT)$~for~the set of inertial classes of dis\-cre\-te series 
representations
of $\G$ with endo-class $\TT$.
The local Jac\-quet-Langlands correspondence~\eqref{LJLC} thus
induces a bijective map
\begin{equation*}
\boldsymbol{\pi}_0 : \Dd_0(\G,\TT) \to \Dd_0(\H,\TT).
\end{equation*}
The cuspidal support induces a bijection between $\Dd_0(\G,\TT)$ 
and the set of inertial classes of~simple supercuspidal pairs of $\G$ 
with endo-class $\TT$. 
Thanks to~\eqref{PARISSC}
the~sets $\Dd_0(\G,\TT)$~and $\Dd_0(\H,\TT)$ are 
parametrized by the set $\Ga\backslash\X$ of 
$\Ga$-orbits of characters of $\kk^\times$.
The bijection~$\boldsymbol{\pi}_0$~thus turns into a permutation 
$\Upsilon$ of $\Ga\backslash\X$, that we would like to describe. 
The purpose of Propo\-sition~\ref{ExpliUpsi} below is to show that, 
in a certain sense, by considering various $m\>1$ such that $md$
is divisible by the degree of $\TT$,
one can reduce the computation of $\Upsilon([\car])$ to the case where 
$\car$ is suitably regular.

\subsection{}
\label{vlad}

We fix a simple stratum $[\aa,\b]$ in $\Mat_m(\D)$ such that 
$\bb=\aa\cap\B$ is maximal in $\B$, 
together with~a simple character $\t\in\Cc(\aa,\b)$
with endo-class $\TT$, 
and a $\b$-exten\-sion $\k$ of $\t$.~The integer $m'$ coming from~\eqref{ISOB} is
$m' = m {(d,g)}/{g}$, where $g$ denotes the degree of $\TT$.
Write $\X$ for the group of complex characters of 
$\kk^\times$.
Thanks to Proposition~\ref{LaSURJ} (see also~\eqref{PARISSC})
we have a bijective map
\begin{equation}
\label{bigPI}
\begin{array}{ccc}
\Ga\backslash\X & \to & \Dd_0(\G,\TT) \\
{[\car]} & \mapsto & \Om(\k,\car) 
\end{array}
\end{equation}
where $\Om(\k,\car)$ is the inertial class of discrete series 
representations of $\G$ that contain the simple type $\bl(\car)$.

Similarly, we choose a maximal simple character $\t'\in\Cc(\aa',\b')$
in $\H$ with endo-class $\TT$, 
and a max\-imal $\b$-exten\-sion $\k'$ of $\t'$. 
We get a bijection
$[\car]\mapsto\Om(\k',\car)$ between $\Ga\backslash\X$ and $\Dd_0(\H,\TT)$.

Let $\Upsilon$ be the unique bijective map such that the diagram 
\begin{equation*}
\begin{CD}
\Ga\backslash\X @>{\Upsilon}>> \Ga\backslash\X \\
@VVV @VVV \\
\Dd_0(\G,\TT) @>>{\boldsymbol{\pi}_0}> \Dd_0(\H,\TT) \\
\end{CD}
\end{equation*}
is commutative, where the vertical maps are given by~\eqref{bigPI} and its 
analogue for $\H$. 
It depends~on the choice of the maximal $\b$-extensions $\k$ and $\k'$.
By Proposition~\ref{AdeleGodel} and Corollary~\ref{lLINKEDJLcoro}, 
we have the following fact.

\begin{prop}
\label{Upsilonmodl}
For any prime number $\ell$,
the bijection $\Upsilon$ is compatible with 
taking $\ell$-regu\-lar parts.
More precisely, 
the $\Ga$-orbits of 
$\car,\b\in\X$ have the same $\ell$-regular part 
if~and only if $\Upsilon([\car])$ and 
$\Upsilon([\b])$ have the same $\ell$-regular part. 
\end{prop}

Proposition~\ref{Upsilonmodl} suggests that,
with a suitable choice of $\ell$,
it may be possible to deduce $\Upsilon([\car])$ from the knowledge of 
$\Upsilon([\b])$.
We will illustrate this idea in Proposition~\ref{ExpliUpsi} below.

\subsection{}

We first give another property of the map $\Upsilon$. 
Set $n'=n/g=m'd'$.
Given $\car\in\X$, let $f$ be the cardinality of its $\Ga$-orbit, 
and write
\begin{equation}
\label{DEFSDA}
s(\car) = s_\D([\car]) = \frac{d'}{(f,d')}.
\end{equation}
Recall that $d'$ is the degree of $\dd$ over $\ee$ (the residue field of 
$\F[\b]$), 
thus we have $d'=d/(d,g)$. 
Note that the cardinality of its $\Gal({\kk}/\dd)$-orbit is equal to 
$f/(f,d')$, which was denoted by $\u$ in paragraph~\ref{Par5}.

\begin{defi}
We call the integer $f$ the \emph{parametric degree} of $\car\in\X$. 
\end{defi}

This is related to the notion of parametric degree for a discrete series 
representation as follows: 
any discrete series representation in 
$\Om(\k,\car)$ has parametric degree $fg$. 

Since the local Jacquet Langlands correspondence preser\-ves the parametric
degree~\cite{BHJL3} we have the following result. 

\begin{lemm}
\label{conspdf}
For all $\car\in\X$, the parametric degrees of $[\car]$ and $\Upsilon([\car])$ 
are equal.
\end{lemm}

Note that the set $\Om(\k,\car)$~is made of cuspidal representations 
with cuspidal Jacquet--Lang\-lands transfers if and only if 
$f=n'$, 
that is, if and only if $\car$ is $\ee$-regular. 
Indeed, from~\cite{BHJL3}, a discrete series representation of $\G$ is cuspidal 
with cuspidal Jacquet--Lang\-lands transfer if and only if its 
parametric degree is $n$.

\subsection{}
\label{nextss}

Let $a\>1$ be a positive integer.
We consider the simple stratum $[\aa^{*},\b]$ in 
$\Mat_{am}(\D)$,~where~$\aa^{\dag}$~is the hereditary order $\Mat_a(\aa)$, 
and write $\t^{\dag}\in\Cc(\aa^{\dag},\b)$ for the transfer of $\t$.
Associated with 
$\k$, there is a coherent choice of a maximal 
$\b$-extension $\k^{\dag}$ of the simple character $\t^{\dag}$ 
(\cite[Remarque~5.17]{MSt}).
We fix~a finite extension $\kk^{\dag}$ of $\kk$ of degree $a$ and write $\X^{\dag}$ 
for the group of complex characters~of $\kk^{\dag\times}$.
Repeating the arguments of paragraph~\ref{vlad} with~$\GL_{am}(\D)$ and~$\GL_{an}(\F)$, we get a bijective map 
$\Upsilon^{\dag}:\Ga\backslash\X^{\dag}\to\Ga\backslash\X^{\dag}$. 
We have the following straightforward result. 

\begin{lemm}
\label{LEmmaUp}
Let $[\car]\in\Ga\backslash\X$, and let $\L(\rho,r)$ be in the inertial 
class $\Om(\k,\car)$, 
for some integer~$r$ dividing $m$ and 
cuspidal representation $\rho$ of $\GL_{m/r}(\D)$.
Then $\L(\rho,ar)$ is in the inertial class $\Om(\k^\dag,\car^{\dag})$,
where $\car^{\dag}$ is the character $\car\circ\N_{\kk^{\dag}/\kk}$ of 
$\kk^{\dag\times}$. 
\end{lemm}

\begin{proof}
With the notation of paragraph~\ref{Par5}
and writing $\M$ for the Levi subgroup $\G\tdt\G\subseteq\GL_{am}(\D)$
{and $\U$ for the unipotent radical of the parabolic subgroup made of 
upper $a\times a$~block triangular matrices of $\GL_{am}(\D)$, 
this follows from the fact that the representation of 
$\J(\aa^\dag_r,\b)\cap\M$ on the 
$\J(\aa^\dag_r,\b)\cap\U$-invariant 
subspace of the transfer $\k^\dag_{ar}$ of~$\k^\dag$ to $\J(\aa^\dag_{ar},\b)$ 
is $\k\odo\k$.}
\end{proof}

For $\car\in\X$, the orbit $[\car^{\dag}]$ depends only on 
$[\car]$, and we denote it $[\car]^{\dag}$.
By Lemma~\ref{LEmmaUp} we thus have 
\begin{equation*}
\Upsilon^{\dag}([\car^{\dag}])=\Upsilon(\car)^{\dag}
\end{equation*}
for any character $\car\in\X$.

Given $\car\in\X$, we write $f$ for its parametric degree, 
and $\ee[\car]$ for the subfield of $\kk$ of 
degree $f$ over $\ee$.

\begin{lemm}
\label{reds1}
Let $\a\in\X$.
There are an integer $a\>1$,
a prime number $\ell\neq p$ 
not dividing the order of $\ee[\car]^\times$
and an $\ee$-regular character $\b\in\X^{\dag}$ 
such that $\b\equiv\a^{\dag}$ mod $\ell$. 
\end{lemm}

\begin{proof}
First recall the following result, known as Zsigmondy's Theorem~\cite{Zsig}. 

\begin{lemm}
\label{ZSIG}
Let $b,r\>2$ be integers.
There exists a prime number $\ell$ which divides $b^r-1$~but not 
$b^{i}-1$ for any $i\in\{1,\dots,r-1\}$, 
except when $r=6$ and $b=2$, 
and when $r=2$ and $b=2^k-1$ for some $k\>1$.
\end{lemm}

Let us write $\Q$ for the cardinality of $\ee$,
and let us fix an $a\>1$ such that $an'>6f$.
Applying Lemma~\ref{ZSIG} with $b=\Q^f$ and $r=an'/f$, 
we obtain a~prime number $\ell$ dividing 
$b^{r}-1$ but not dividing $b^{i}-1$ for any $i\in\{1,\dots,r-1\}$.
It follows that $b$ has order $r$ in the group $(\ZZ/\ell\ZZ)^\times$. 

Let $\xi$ be a nontrivial character of $\kk^{\dag\times}$ of order $\ell$. 
Then the character $\b=\xi\car^{\dag}$ is congruent to $\a^{\dag}$ mod $\ell$. 
Since the order of $\a$ is prime to $\ell$ (for it divides $b-1$), 
the cardinality of the $\Ga$-orbit of $\b$ is the least common multiple of $f$ 
and the order of $\Q$ in $(\ZZ/\ell\ZZ)^\times$.
This cardinality is equal to $fr=an'$, thus $\b$ is $\ee$-regular. 
\end{proof}

\begin{rema}
\label{pifou}
\begin{enumerate}
\item
It is not always possible to choose $a=1$.
For instance, if $\car$ is trivial, 
$n'=2$ and $\ee$ has $7$ elements, 
then no prime number $\ell$ satisfies the required condition.
We thank Guy Henniart for a suggestion that brought us to introduce 
the process described here.
\item
\label{lnot1}
The proof of Lemma~\ref{reds1} shows that,
for any character $\car\in\X$, we can choose $a$ to be any integer $\>7$.
Moreover, the choice of $a$ and $\ell$ only depend on the parametric degree $f$,
not on $\car$.
\item
Note that $\ell$ cannot be $2$.
Indeed we have $\ell\neq p$ and, 
if $p$ is odd, then the fact that $\ell$ does not divide $\Q^{f}-1$ 
(the order of $\ee[\car]^\times$)
implies that $\ell\neq 2$.
\end{enumerate}
\end{rema}

With the notation of Lemma~\ref{reds1}, we get the following result. 

\begin{prop}
\label{ExpliUpsi}
Assume that $\Upsilon^{\dag}([\b])=[\b\mu]$ for some character $\mu\in\X^{\dag}$.
Then $\mu_\ell=\nu^{\dag}$ for~some character $\nu\in\X$ and we have 
$\Upsilon([\car])=[\car\nu]$.
\end{prop}

\begin{proof}
Let us write $\Upsilon([\car])=[\car']$ for some $\car'\in\X$.
Then $[\car'^{\dag}]\equiv[\b\mu]$ mod $\ell$.
By Lemma~\ref{conspdf},
the parametric degree of $\car'$ is $f$, thus $\ee[\car']=\ee[\car]$.
It follows that $\ell$ does not divide the 
order of $\car'$.
Write $\b=\xi\car^{\dag}$
for some character $\xi$ whose order is a power of 
$\ell$. 
Taking $\ell$-regular~parts,~we get 
$[\car'^{\dag}]_\ell^{}=[\car'^{\dag}]=[\car^{\dag}\mu^{}_\ell]$.
Changing $\car'$ in its $\Ga$-orbit, 
we may assume that $\car'^{\dag}=\car^{\dag}\mu_\ell^{}$.
Thus $\mu_\ell=\nu^{\dag}$ for~some $\nu\in\X$.
Since $\N_{\kk^{\dag}/\kk}$ is surjective, 
we get $\Upsilon([\car])=[\car\nu]$.
\end{proof}

\section{The essentially tame case}
\label{SECTION9}

The purpose of this section is to illustrate Proposition~\ref{ExpliUpsi} in 
the essentially tame case. 
Assume that $\TT$ is essentially tame.
We thus have $\boldsymbol{\pi}_1(\TT)=\TT$ 
(see Corollary~\ref{coroesstame}).
With the notation of paragraph~\ref{vlad}, 
we let $\k$ be the unique $\b$-extension 
of $\t$ such that  
$\det(\k)$ has order a power of $p$, and similarly for $\k'$. 
We thus get a canonical permutation $\Upsilon$ of the finite set 
$\Ga\backslash\X$. 

Our purpose is to use~\cite[Theorem~A]{BHJL3}
together with Proposition~\ref{ExpliUpsi}
in order to get a formula for $\Upsilon([\car])$ for \emph{all} characters 
$\car\in\X$, 
and not only for those $\car$ giving rise to 
cuspidal representations of $\H$, that is, $\ee$-regular ones.
The main result of this section is the following theorem.

\begin{theo}
\label{MT9}
There is a canonically determined character $\mu$ of $\kk^\times$, 
depending only on $m$, $d$ and $\TT$,
such that $\mu^2=1$ and
\begin{equation}
\label{SCT}
\Upsilon([\car]) = [\car\mu]
\end{equation}
for all characters $\car\in\X$.
\end{theo}

More precisely, we will see that 
the character $\mu$ is the ``rectifier'' given by 
Bushnell--Henniart's First Compari\-son Theorem~\cite[6.1]{BHJL3}
together with~\cite[Corollary~6.9 and~(6.7.4)]{BHJL3}.

Let us write $\XR$ for the set of $\ee$-regular characters in $\X$.
We will first~prove that~\eqref{SCT}~holds for all characters $\car\in\XR$, 
using~\cite[Theorem~A]{BHJL3} (see Paragraph~\ref{JuanAsensio}).
Since this theo\-rem~is~for\-mula\-ted in terms of admissible pairs,
we will have to translate~it~in terms of our $\car$-para\-meters.

\subsection{}

We first recall the definition of admissible pairs~\cite{Howe,BHJL3}, 
and basic facts about them.

\begin{defi}
An \emph{admissible pair} is a pair $(\L/\F,\xi)$ made of a finite, 
tamely ramified field extension $\L/\F$ and a character $\xi$ of $\L^\times$ 
such that:
\begin{enumerate}
\item
$\xi$ does not factor through $\N_{\L/\K}$ for any field $\K$ such that 
$\F\subseteq\K\subsetneq\L$;
\item
if the restriction of $\xi$ to the $1$-units $1+\p_\L$ factors through 
$\N_{\L/\K}$ for some field $\K$ such that 
$\F\subseteq\K\subseteq\L$, then $\L/\K$ is unramified. 
\end{enumerate}
\end{defi}

Two admissible pairs $(\L_i/\F,\xi_i)$, $i=1,2$, are said to be 
\emph{isomorphic} if there is an $\F$-isomorphism $\phi:\L_2\to\L_1$ 
such that $\xi_2=\xi_1\circ\phi$.
The \emph{degree} of an admissible pair $(\L/\F,\xi)$ is $[\L:\F]$.
We also introduce the following definition, which will be convenient to us. 

\begin{defi}
Two admissible pairs $(\L_i/\F,\xi_i)$ for $i=1,2$, are said to be 
\emph{inertially equivalent} if there are an unramified character $\chi$ of 
$\L_2^\times$ and an isomorphism $\phi:\L_2\to\L_1$ of extensions of $\F$
such that $\chi\xi_2=\xi_1\circ\phi$.
We will write $[\L_1/\F,\xi_1]$ for the inertial class of $(\L_1/\F,\xi_1)$. 
\end{defi}

Let $(\L/\F,\xi)$ be an admissible pair.
By~\cite[4.1~Lemma]{BHJL3}, there is a unique sub-extension $\P/\F$ 
of $\L/\F$ such that $\xi\ |\ 1+\p_\L$ factors through the norm $\N_{\L/\P}$
and which is minimal for~this property. 
It is called the \emph{parameter field} of the admissible pair.
Then $\L/\P$ is unramified and, 
if we write $\xi\ |\ 1+\p_\L=\xi_1\circ\N_{\L/\P}$ for some character~$\xi_1$ 
of $1+\p_\P$, then 
$(\P/\F,\xi_1)$~is an admissible $1$-pair in the sense of 
\cite[3.3]{BHJL3}.
As in~\cite[4.3]{BHETLC3}, the admissible $1$-pair $(\P/\F,\xi_1)$ 
determines an endo-class.
This endo-class is essentially tame with parameter field $\P/\F$, 
in the sense of~\cite[2.4]{BHToAnEffectiveLC}.
The para\-meter field $\P$ and the 
endo-class only depend on the inertial class $[\L/\F,\xi]$.

\begin{rema}
\label{PE}
Suppose that $(\L/\F,\xi)$ has associated endo-class $\TT$, 
the one that we have fixed at the beginning of this section. 
Then $\P$ is $\F$-isomorphic to $\E$.
Moreover, by changing $(\L/\F,\xi)$ in its isomorphism class~\cite[4.1]{BHJL3}, 
we may assume that $\P$ is equal to $\E$.
\end{rema}

\subsection{}

Now let $(\L/\F,\xi)$ be an admissible pair with associated 
endo-class $\TT$ and degree $t$ dividing $n$.
By~\cite[4.3~Lemma~1]{BHJL3},~there is 
a unique character $\xw$ of the group $\Oo_\L^\times$ of units
of $\Oo_\L$ such that: 
\begin{enumerate}
\item
the characters $\xw$ and $\xi$ coincide on the principal unit subgroup 
$1+\p_\L$;
\item
the order of $\xw$ is a power of $p$.
\end{enumerate}
The character $\xi\cdot\xw^{-1}$ of $\Oo_\L^\times$ is tamely ramified, 
thus induces a character of $\ll^\times$, denoted $\xt$, 
where $\ll$ is the residue field of $\L$.
This character $\xt$ depends only on the inertial class of $(\L/\F,\xi)$.
Moreover, since $(\L/\F,\xi)$ is an admissible pair, $\xt$ is $\ee$-regular. 
Let~us fix an $\ee$-embedding of $\ll$ in $\kk$ and write 
$\car_\xi$ for the character $\xt\circ\N_{\kk/\ll}$ of $\kk^\times$.
Its $\Ga$-conju\-ga\-cy class does not depend on the choice of the 
embedding of $\ll$ in $\kk$, and 
its parametric degree $f$
is equal to $[\L:\E]$. 
We thus have $t=fg$.

We write $\Pp_n(\TT)$ for the set of isomorphism classes of admissible pairs 
with endo-class $\TT$ and degree dividing $n$.

\begin{lemm}
\label{delaMole}
\begin{enumerate}
\item
The character $\a_\xi$ is $\ee$-regular if and only if $[\L:\F]=n$.
\item
The map
\begin{equation}
\label{APGCA}
[\L/\F,\xi] \mapsto [\car_\xi]
\end{equation}
induces a bijection between 
the set of inertial classes of admissible pairs in $\Pp_n(\TT)$ 
and $\Ga\backslash\X$. 
\end{enumerate}
\end{lemm}

\begin{rema}
Note that the map~\eqref{APGCA} is well-defined, thanks to Remark~\ref{PE}.
\end{rema}

\begin{proof}
The character $\a_\xi$ is $\ee$-regular if and only if $f=n'$.
Multiplying by $g$, this is equivalent to $t=n$.
This gives us the first part of the lemma.

Given $\car\in\X$, there is a uniquely determined field $\ll$ such that 
$\ee\subseteq\ll\subseteq\kk$ and $\car$ factors through the norm 
$\N_{\kk/\ll}$, and which is minimal for this property.
Write $\car=\b\circ\N_{\kk/\ll}$ for some~character~$\b$ of 
$\ll^\times$, which is $\ee$-regular by minimality of $\ll$.
Let $\L$ be an unramified extension of $\E$ with residue
field $\ll$.
Then~$\b$ inflates to a tamely ramified character of the units subgroup of 
$\L$, still denoted~$\b$.
Now let $\xi$ be any character of $\L^\times$ extending $\xw\b$.
Since the character $\b$ is $\ee$-regular, it follows that 
the pair $(\L/\F,\xi)$ is ad\-mis\-si\-ble.
The $\Ga$-orbit $[\car_\xi]$ 
associated with its inertial class is equal to $[\car]$.
The map~\eqref{APGCA} is thus surjective. 

We now assume that we have two admissible pairs 
$(\L_i/\F,\xi_i)$ for $i=1,2$, with same image~$[\car]$ in $\Ga\backslash\X$. 
For each $i$, we may assume that the parameter field of $(\L_i/\F,\xi_i)$ is 
$\E$ by Remark~\ref{PE}.~The 
character $\xi_i\ |\ 1+\p_{\L_i}$ thus factors through 
$\N_{\L_i/\E}$ and $\E$ is minimal for this property.
We have an $\ee$-regular character $\xi_{i,{{\rm t}}}$ of 
$\ll_i^\times$, 
where $\ll_i$ is the residue field of $\L_i$.
Since the $\Ga$-orbits $[\car_{\xi_1}]$ and $[\car_{\xi_2}]$ are 
equal, they have the same cardinality $f$.
The fields $\ll_1$, $\ll_2$ thus have the same degree over $\ee$,
and $\L_1$, $\L_2$ have the same degree $f$ over $\E$. 
We thus may assume 
(by~changing~the pair $(\L_2/\F,\xi_2)$ in its isomorphism class)
that $\L_1=\L_2$, denoted $\L$.
We now have two 
characters $\xi_{1,{{\rm t}}}$ and $\xi_{2,{{\rm t}}}$ 
of $\ll^\times$, which are conjugate under $\Gal(\ll/\ee)$.
Changing again $(\L_2/\F,\xi_2)$ in its isomorphism class, 
we may assume that they are equal.
Thus the admissible pairs 
$(\L_i/\F,\xi_i)$, for $i=1,2$, are inertially equivalent. 
\end{proof}

\subsection{}
\label{JuanAsensio}

The Parametrization Theorem~\cite[6.1]{BHJL3} gives us a bijective map
\begin{equation}
\label{BHPTmap}
(\L/\F,\xi)\mapsto\Pi(\G,\xi)
\end{equation}
between $\F$-isomorphism classes of admissible pairs of degree $n$ 
and isomorphism classes of essen\-tially tame irreducible cuspidal 
representations 
of $\G$, that is, cuspidal representations with essen\-tially tame 
endo-class and parametric degree $n$.

More precisely, by examining~\cite[4.2 and~4.3]{BHJL3}, we see that 
the cuspidal 
representation $\Pi(\G,\xi)$ associa\-ted with an admissible pair 
$(\L/\F,\xi)$ of degree $n$ and endo-class $\TT$
contains the maximal simple type $\k\otimes\s$,
where $\k$ is the maximal $\b$-extension fixed at the beginning 
of this section 
and $\s$ is the irreducible cuspidal representation of $\GA$ with 
Green parameter $\xt$.

\begin{lemm}
\label{Marcheschi}
\begin{enumerate}
\item
Given an admissible pair $(\L/\F,\xi)$ of degree $n$ and endo-class $\TT$, 
the~ir\-reducible 
cus\-pidal representation $\Pi(\G,\xi)$ belongs to the inertial class 
$\Om(\k,\car_\xi)$. 
\item
The bijection~\eqref{BHPTmap} induces a bijection
between inertial classes of admissible pairs of degree $n$ 
and inertial classes of essentially tame cuspidal representations of $\G$.
\end{enumerate}
\end{lemm}

\begin{proof}
Comparing with the construction in paragraph~\ref{Par5}, the maximal 
simple type $\k\otimes\s$ above is $\bl(\car_\xi)$.
Note that $\car_\xi$ is equal to $\xt$ since $\L/\F$ has degree $n$.
This gives us the first part of the lemma. 

An inertial class of essentially tame cuspidal representations of $\G$ 
with endo-class $\TT$ has the form $\Om(\k,\car)$ for some $\car\in\XR$. 
The second part of the lemma thus follows from Lemma~\ref{delaMole}.
\end{proof}

We now prove Theorem~\ref{MT9} for $\ee$-regular characters. 

\begin{prop}
\label{firststep}
\begin{enumerate}
\item
There is a canonically determined character $\mu\in\X$,
depending only on $m$, $d$ and $\TT$, 
such that $\mu^2=1$ and
\begin{equation*}
\Upsilon([\car]) = [\car\mu]
\end{equation*}
for all characters $\car\in\XR$.
\item
The character $\mu$ is non-trivial if and only if $p\neq2$ and the 
integer 
\begin{equation*}
y(\TT,m,d) = m(d-1)+m'(d'-1)+u(v-1)
\end{equation*}
is odd, where the integers $u,v\>1$ are defined by
$uv=n/w$, $v=d/(d,w)$ with $w=n/e({\E/\F})$.
\end{enumerate}
\end{prop}

\begin{proof}
Let $\car\in\XR$ and let $(\L/\F,\xi)$ be an admissible pair 
of degree $n$ and endo-class $\TT$~whose inertial class 
is associated with $[\car]$. 
By~\cite[Theorem~A]{BHJL3}, there is a
tamely ramified character $\nu$~of $\L^\times$
such that $(\L/\F,\xi\nu)$ is admissible, 
$\nu^2$ is trivial and 
the Jacquet--Langlands transfer of $\Pi(\G,\xi)$ 
is $\Pi(\H,\xi\nu)$.
Since $\L/\F$ has degree $n$ and $\L$ is unramified over $\E$, 
the residue field~of~$\L$ identifies with $\kk$. 
Write $\mu$ for the character of $\kk^\times$ induced by 
the restriction of $\nu$ to the units subgroup of $\L$.
This character is entirely described by~\cite[Corollary~6.9]{BHJL3},
which gives us Assertion (2).

Taking inertial classes and using Lemma~\ref{Marcheschi},
the Jacquet--Langlands correspondence matches together 
the inertial class $\Om(\k,\a)$ 
of $\Pi(\G,\xi)$ with that of $\Pi(\H,\xi\nu)$,
and the latter can be written $\Om(\k',\car')$ for 
$[\car']=[\car_{\xi\nu}]=[\car\mu]$.
The result follows.
\end{proof}

\subsection{}

We now prove Theorem~\ref{MT9}.
Let us fix an \emph{odd} integer $a\>7$
(see Remark~\ref{pifou}(\ref{lnot1})).
We will see below why it is~con\-ve\-nient to choose $a$ odd.
We use the notation introduced in para\-graph~\ref{nextss}.
In particular, we~have $\b$-extensions $\k^\dag$, $\k^{\prime\dag}$
and a permutation $\Upsilon^\dag$ of $\Ga\backslash\X^\dag$.
We must pay attention to the fact that the determinants of 
$\k^\dag$ and $\k^{\prime\dag}$ have orders which may not be powers of 
$p$, thus Proposition~\ref{firststep} may not apply to $\Upsilon^\dag$ directly. 

Let us write $\varkappa^\dag$ for the $\b$-extension on 
$\J(\aa^\dag,\b)$ 
whose determinant has order a power of $p$.~By 
Remark~\ref{depkappa} there is a character $\xi$ of $\J(\aa^\dag,\b)$ 
trivial on $\J^1(\aa^\dag,\b)$ such that 
$\varkappa^\dag=\k^\dag\xi$.
It induces a character of $\GA$ of the form 
$\chi\circ\N_{\dd/\ee}\circ\det$ for some character $\chi$ of $\ee^\times$. 

Similarly, we have a $\b$-extension $\varkappa'^\dag$ 
whose determinant has order a power of $p$,
and characters 
$\xi',\chi'$ such that $\varkappa'^\dag=\k'^\dag\xi'$ and $\xi'$
induces the character 
$\chi'\circ\N_{\dd/\ee}\circ\det$
of $\GA$.
We write $\Phi$ for~the permutation of $\Ga\backslash\X^\dag$ 
corresponding to the $\b$-extensions $\varkappa^\dag$ and $\varkappa'^\dag$. 
We write $\d$ for the~char\-acter 
$(\chi'\chi^{-1})\circ\N_{\kk^\dag/\ee}\in\X^\dag$.

\begin{lemm}
The character $\d$ is trivial.
\end{lemm}

\begin{proof}
Let $\b\in\X^\dag$ be $\ee$-regular.
Applying~Proposi\-tion 
\ref{firststep} to $\Phi$ gives us 
$\Phi([\b])=[\b\l]$, where $\l\in\X^\dag$ is the rectifying character
corresponding to $am$, $d$ and $\TT$. 
It has order at most $2$, and it is non-trivial if and only if $p\neq2$ 
and the integer
\begin{equation*}
y(\TT,am,d) = a\cdot y(\TT,m,d)
\end{equation*}
is odd.
Since $a$ is odd, it follows that $\l$ is trivial if and only if $\mu$ is, 
that is $\l=\mu^{\dag}$.
Now let $\e$~be the character $\chi\circ\N_{\kk^\dag/\ee}$, 
and define $\e'$ similarly. 
Comparing $\Phi$ and $\Upsilon^\dag$ thanks to Remark~\ref{depkappa},
we get $\Upsilon^\dag([\b\e])=[\b\mu^\dag\e']$ for all $\ee$-regular $\b\in\X^\dag$. 
Since $\b\e^{-1}$ is $\ee$-regular if and only if $\b$ is, this gives us
\begin{equation}
\label{deMouy}
\Upsilon^\dag([\b])=[\b\d\mu^\dag]
\end{equation}
for all $\ee$-regular $\b\in\X^\dag$. 

Now let $\car\in\XR$. 
By Lemma~\ref{reds1} there are a 
prime number $\ell\neq p$~not dividing the order~of~$\kk^\times$
and an $\ee$-regular character $\b\in\X^{\dag}$ 
such that $\b\equiv\a^{\dag}$ mod~$\ell$. 
By~\eqref{deMouy} and Proposition~\ref{ExpliUpsi}
we~get $\Upsilon([\car])=[\car\nu]$ 
for some $\nu\in\X$ such that $\nu^\dag$ is the $\ell$-regular part of 
$\d\mu^\dag$.
Since $\car$ is $\ee$-regular,
Propo\-si\-tion~\ref{firststep} applied to $\Upsilon$ gives us 
$\Upsilon([\car])=[\car\mu]$.
Putting these equalities together, we get
\begin{equation*}
\text{$[\car^\dag\mu^\dag] \equiv [\car^\dag\mu^\dag\d]$ mod $\ell$}.
\end{equation*}
The character $\d$ can thus be written 
$\xi(\car^\dag\mu^\dag)^{\Q^i-1}$ for some integer $i\in\{0,\dots,n'-1\}$ 
and some $\xi\in\X^\dag$ whose order is a power of $\ell$.
(Recall that $\Q$ is the cardinality of $\ee$.)
Since $\mu$ has order at most $2$, we get 
$\d=\xi(\car^{\dag})^{\Q^i-1}$.
Since the order of $\d$ divides $\Q-1$, we have
\begin{equation*}
\car^{\dag(\Q^i-1)(\Q-1)}=\xi^{1-\Q}.
\end{equation*}
Since both $\Q-1$ and the order of $\car$ are prime to 
$\ell$, we get $\xi=1$.
Thus the order of $\car$, that we may assume to be $\Q^{n'}-1$
by choosing for $\car$ a generator of $\X$,
divides $(\Q^i-1)(\Q-1)$.
This implies $i=0$, thus $\d$ is trivial as expected.
\end{proof}

\begin{rema}
If $a$ were chosen to be even, the same proof would apply 
(with $\l=1$)
and we would get $\d=\mu^\dag$.
\end{rema}

Now let $\car\in\X$ be arbitrary.
By Lemma~\ref{reds1} there are 
a prime number $\ell\neq p$ 
not dividing~the order of $\ee[\car]^\times$
and an $\ee$-regular character $\b\in\X^{\dag}$ such that 
$\b$ is congruent to $\a^\dag$ mod $\ell$.
Since $\d$ is trivial,~\eqref{deMouy} gives us 
$\Upsilon^\dag([\b])=[\b\mu^\dag]$.
By Proposition~\ref{ExpliUpsi}, 
we have $\Upsilon([\car])=[\car\nu]$ for some character $\nu\in\X$ 
such that 
$\nu^\dag$ is the $\ell$-regular part of $\mu^\dag$.
Thus $\nu^\dag=\mu^\dag$, which implies $\nu=\mu$. 
This completes the proof of Theorem~\ref{MT9}.

\subsection{}

We now translate Theorem~\ref{MT9} in terms of admissible pairs. 
Let $(\L/\F,\xi)$ be an admissible~pair with degree dividing $n$ 
and endo-class $\TT$.
The character class $[\car_\xi]\in\Ga\backslash\X$ given by the map 
\eqref{APGCA} corresponds to an inertial class of discrete 
series representations $\Om(\k,\car_\xi)$.
Let us write $\Pi_0(\G,\xi)$ for this inertial class. 
The map
\begin{equation*}
[\L/\F,\xi] \mapsto \Pi_0(\G,\xi)
\end{equation*}
is a bijection between $\Pp_n(\TT)$ and $\Dd_0(\G,\TT)$.

\begin{theo}
\label{THPAI}
Let $(\L/\F,\xi)$ be an admissible pair with degree dividing $n$. 
There is a canonically determined tamely ramified character 
$\mu$ of the units subgroup of $\L$ such that $\mu^2=1$ 
and
\begin{equation*}
\boldsymbol{\pi}_0(\Pi_0(\G,\xi)) = \Pi_0(\H,\xi\mu).
\end{equation*}
It depends only on $m$, $d$ and the restriction of $\xi$ to the principal 
units $1+\p_\L$.
\end{theo}

Note that by $\Pi_0(\H,\xi\mu)$ we mean the inertial class corresponding to
the pair $[\L/\F,\xi\hat\mu]$ for any choice of extension 
$\hat\mu$ of $\mu$ 
to $\L^\times$; this is independent of the choice of $\hat\mu$. 

\begin{rema}
Let $t$ be the degree of $\L/\F$
and write $s$ for the integer $s(\car_\xi)$ defined by \eqref{DEFSDA}. 
The~para\-metric degree $f=[\L:\E]$ of $\car_\xi$ 
divides $m'd'$.
Hence $\u=f/(f,d')$ divides $m's$,
thus~$m'$.
Let us defi\-ne an integer $r\>1$ by $m'=ur$, 
or equivalently by $n=rst$.
Any discrete series~re\-pre\-sen\-ta\-tion in $\Pi_0(\G,\xi)$ has the form 
$\L(\rho,r)$ for some 
cuspidal representation $\rho$ of $\GL_{m/r}(\D)$ 
with para\-metric degree $t$.
\end{rema}

\providecommand{\bysame}{\leavevmode ---\ }


\end{document}